\documentclass[a4paper, leqno]{article}
\usepackage[utf8]{inputenc}
\usepackage[french,english]{babel}
\usepackage[T1]{fontenc}
\usepackage{amsmath, amsfonts, amssymb, amsthm}
\usepackage{a4wide}
\usepackage{soulutf8}%pour utiliser \ul pour souligner
\usepackage{authblk}%Pour les auteurs
\usepackage{color, soul}

\usepackage[titletoc,toc,title]{appendix}

\usepackage{circuitikz}
\usepackage{tikz}
\usepackage{pgfplots}
\usepackage{verbatim}
\usetikzlibrary{arrows}

\newtheorem{theorem}{Theorem}%[section]
\newtheorem*{remark}{Remark}%[section]
\newtheorem{lemma}{Lemma}
\newtheorem{proposition}{Proposition}

\newtheorem{definition}{Definition}%[section]
\newtheorem*{notation}{Notation}
\newtheorem*{notations}{Notations}

\makeatletter

\@addtoreset{equation}{section}
\makeatother

\usepackage{xcolor}

\def\R{\mathbb{R}}

\def\N{\mathbb{N}}

\usepackage{hyperref}
\hypersetup{ 
colorlinks=false, 
linkbordercolor={1 1 1}, % set to white 
citebordercolor={1 1 1} % set to white 
} 
\usepackage{graphicx}
\usepackage{rotating}
\usepackage{epic}
\usepackage{eepic}
\usepackage{epsfig}
\usepackage{array, multirow, hhline, tabularx, graphicx, wrapfig}

\everymath{\displaystyle}
\usepackage{color}
\usepackage{verbatim}
\usepackage{listings}
\usepackage[font=small,skip=0pt]{caption}
\usepackage[notlof, notlot, nottoc]{tocbibind}

\begin{document}

\author{Vincent Renault, Michèle Thieullen\thanks{Sorbonne Universit\'es, UPMC Univ Paris 06, CNRS UMR 7599, Laboratoire de Probabilit\'es et Mod\`eles Al\'eatoires, F-75005, Paris, France. email: \href{mailto:vincent.renault@upmc.fr}{vincent.renault@upmc.fr}, \href{mailto:michele.thieullen@upmc.fr}{michele.thieullen@upmc.fr}}, Emmanuel Trélat\thanks{Sorbonne Universit\'es, UPMC Univ Paris 06, CNRS UMR 7598, Laboratoire Jacques-Louis Lions, Institut Universitaire de France, F-75005, Paris, France. email: \href{mailto:emmanuel.trelat@upmc.fr}{emmanuel.trelat@upmc.fr}}}

\title{Minimal time spiking in various ChR2-controlled neuron models}

\date{}

%\author[1]{Vincent Renault\thanks{vincent.renault@upmc.fr}}
%\author[1]{Michèle Thieullen\thanks{michele.thieullen@upmc.fr}}
%\author[2]{Emmanuel Trélat\thanks{emmanuel.trelat@upmc.fr}}
%\affil[1]{Laboratoire de Probabilités et Modèles Aléatoires, UPMC}
%\affil[2]{Laboratoire Jacques-Louis Lions, UPMC}

\maketitle

%\tableofcontents

\begin{abstract}
We use conductance based neuron models and the mathematical modeling of Optogenetics to define controlled neuron models and we address the minimal time control of these affine systems for the first spike from equilibrium. We apply tools of geometric optimal control theory to study singular extremals and we implement a direct method to compute optimal controls. When the system is too large to theoretically investigate the existence of singular optimal controls, we observe numerically the optimal bang-bang controls. 



\end{abstract}

\section*{Introduction}

In this paper we investigate, theoretically and numerically, the minimal time control, via Optogenetics, of some widely used finite-dimensional deterministic neuron models such as the Hodgkin-Huxley model (\cite{HH}), the Morris-Lecar model (\cite{MorrisLecar}) and the FitzHugh-Nagumo model (\cite{FitzHugh}). Control of neuron models has been addressed in the literature in different ways. One popular way to investigate this problem is to look at phase reductions of non-linear evolution systems, consisting in reducing the system of equations to a single first-order differential equation, with for essential goal numerical computations of the dynamic programming formulation of the problem (\cite{PhaseReduction}, \cite{Nabi}). Integrate-and-fire models, which are also a simplification of nonlinear sytems to single first-order linear differential equations, receiving stochastic inputs, have been studied in \cite{Feng} in order to minimize the variance of the membrane potential, arguably linked to the variance of the final time, while reaching a given membrane potential threshold in fixed time. Theses simplifications allow the authors to obtain a nice analytic expression for the optimal control. A stochastic integrate-and-fire models have also been used in \cite{Ditlevsen} to find an optimal electrical stimulation to spike in a desired time, a problem close to ours, with numerical computation purposes.

All these studies were exclusively based on control via electrical stimulation. Optogenetics allows a control of excitable cells of a different nature. This recent and thriving technique is based on light stimulation (\cite{methodOfTheYear},\cite{OptoFuture},\cite{Opto10Years}). It has for cornerstone the genetical modification of excitable cells for them to express new ion channels whose opening and closing are triggered by the absorption of photons. In particular, it is able to target specific populations of neurons. Indeed, by designing viruses that will aim at these populations only, the light stimulation will have no effect on the other populations that do not express the new ion channels. This makes Optogenetics a noninvasive technique, in contrast to electrical stimulation that reaches a whole volume of tissue, regardless of the types of neurons that populate this volume. Furthermore, optical devices such as optic fibers and lasers allow to reach deeply embedded populations of neurons. It then provides Optogenetics with a tremendous advantage over electrical stimulation in the exploration of neural tissues and neural functions. The risk of tissue damage is also decreased with this technique. The perspectives of applications in medicine are thus colossal with, among others, the promise to help understand and treat Alzheimer's disease (\cite{Alzheimer}), Parkinson's disease (\cite{Parkinson}), epilepsy (\cite{Epilepsy}), vision loss (\cite{Blindness}), narcolepsy (\cite{Narcolepsy}) and even depression (\cite{Depression}).  

Our work is based on one of these light-gated ion channels called \textit{Channelrhodopsin} (ChR2). It is a depolarizing non-selective cation channel that opens upon a stimulation with blue light. One of the neural events that contains a lot of information is the latency time between two consecutive \textit{action potentials} or \textit{spikes} (a large depolarization of the membrane potential when it goes beyond some threshold). Here we want to specifically address the time optimal control of the first spike in various neuron models, for two different mathematical models of ChR2 introduced in \cite{Nikolic}. Indeed, the mathematical formulation of this problem is really close to the one of the optimal control of the latency time between two spikes. In particular, the investigation of singular trajectories is the same. To the best of our knowledge, this optimal control problem has never been studied before, neither in terms of electrical stimulation, nor in terms of light stimulation. 

In Section \ref{SectionPreliminaries} we set the mathematical framework of conductance-based neuron models and we recall some results of minimal time control problems for affine control systems, and the role of singular controls. We then present in Section \ref{resultsSection} the mathematical model of ChR2 and how the resulting models can be incorportated in conductance-based models. We apply our results to various neuron models in Section \ref{SectionApplication}. For the ChR2-3-state model, we prove that there are no singular optimal controls for two-dimensional models (FitzHugh-Nagumo, Morris-Lecar, reduced Hodgkin-Huxley models) and we give the expression of the \textit{bang-bang} optimal control. We illustrate these results with numerical computations of the optimal controls by means of a direct method. For the ChR2-
4-states model, we numerically observe optimal bang-bang controls. Along the review of the different models, we insist on how optimal control appears as a great tool to discuss and compare neuron models. In particular, it emphasizes a peculiar behavior of the Morris-Lecar model, compared to the other ones, and gives a new argument in favor of the reduced Hodgkin-Huxley model.

Although we focus in this paper on neuron models, our treatment of conductance-based model can be applied to any excitable cells such as cardiac cells for example (see \cite{Abilez} for a work on application of Optogenetics in cardiac cells for simulation purposes).

\section{Preliminaries}\label{SectionPreliminaries}
\subsection{Conductance based models}

Conductance based models form a popular class of simple biophysical models used to represent the activity of an excitable cell, such as a neuron or a cardiac cell. The principle is to give an equivalent circuit representation of the cell by assigning an electrical component to each meaningful biological component of the cell. Finite-dimensional conductance-based models represent the cell as a single isopotential electrical compartment. The lipid bilayer membrane of the cell is represented by a capacitance $C>0$. Across the membrane are disposed voltage-gated ion channels, represented by conductances $g_x>0$ whose values depend on the type $x$ of the channel. An ion channel is a protein that constitutes a gate across the membrane. It has the ability to let ions flow across the membrane or to prevent them from doing so. Ion channels are said selective in the sense that they act as a filter of certain types of ions. The main types of ion channels are potassium ($K^+$) channels, sodium ($Na^+$) channels and calcium ($Ca^{2+}$) channels. The ion flows are driven by electrochemical gradients represented by batteries whose voltages $E_x\in\R$ equal the membrane potential corresponding to the absence of ion flow of type $x$. For that matter, they are called equilibrium potentials. The sign of the difference between the membrane potential and $E_x$ gives the direction of the driving force. The channels are all called voltage-gated because their opening and closing depend on the potential difference across the membrane. This means that the conductances $g_x$ are variable conductances, depending on the membrane potential.

\medskip

The ion flow across the membrane generates an electrical current in the circuit, the possible movements of ions inside the cell being neglected. To each type $x$ of ion channels is associated a macroscopic ion current $I_x$. The total membrane current is the sum of the capacitive current and all of ionic currents considered. In all models we consider in this paper, the ionic currents include a leakage current that accounts for the passive flow of some other ions across the membrane. This current is associated to a fixed conductance $g_L$ and is always denoted by $I_L$. 
 
Every macroscopic ion current $I_x$ is the result of the ion flow through all the ion channels of type $x$. Since the number of ion channels in an excitable cell is very large, the macroscopic conductance $g_x$ is a function of the probability $n_x\in[0,1]$ that a channel of type $x$ opens. In fact, the channels of type $x$ are constituted by several subpopulations of gates that have different dynamics. For that matter, let $k_x\in\N^*$ be the number of subpopulations of the channels of type $x$ and write $(n_{x_1},\dots,n_{x_{k_x}})\in[0,1]^{k_x}$ the probabilities that each gate of the subpopulation opens, that is, $n_{x_i}$ represents the probability that a gate of type $x_i$ opens. The time evolution of these probabilities in each subpopulation depends on the membrane potential and is of first order. For $i\in\{1,\dots,k_x\}$, it is represented on Figure \ref{IonChannel} and the dynamical system governing $n_{x_i}$ is the following

\begin{equation}\label{IonChannelDynamics}
\dot{n}_{x_i}(t) = \alpha_{x_i}(V)(1-n_{x_i}) - \beta_{x_i}(V)n_{x_i},
\end{equation}
 
where $\alpha_{x_i}$ and $\beta_{x_i}$ are smooth functions of the membrane potential $V$.

\begin{figure}[!ht]
\begin{center}
\begin{tikzpicture}[->,>=stealth',shorten >=1pt,auto,node distance=3cm,
  thick,main node/.style={circle,draw,font=\sffamily\Large\bfseries}]

  \node[main node] (1) {C};
  \node[main node] (2) [right of=1] {O};

  \path[every node/.style={font=\sffamily\small}]
    (1) edge [bend left] node [above] {$\alpha_{x_i}(V)$} (2)
    (2) edge  [bend left] node [below]  {$\beta_{x_i}(V)$} (1);

\end{tikzpicture}
\end{center}
\caption{Ion channel of type ${x_i}$}
\label{IonChannel}
\end{figure}
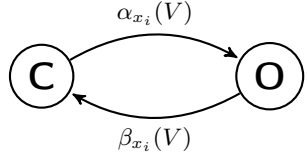

This dynamics can be easily interpreted as follows : when the potential across the membrane is equal to $V$, ion channels in the subpopulation of type $x_i$ open at rate $\alpha_{x_i}(V)$ and close at rate $\beta_{x_i}(V)$.

The macroscopic conductance $g_x$ is then given by

\[
g_x(n_x) = \bar{g}_x f_x(n_{x_1},\dots,n_{x_{k_x}}),
\] 

where $\bar{g}_x$ is the maximum conductance of the channel (i.e., the conductance when all the channels of type $x$ are open) and $f_x$ is a smooth function depending on the type of the channel.  

The macroscopic current $I_x$ of type $x$ is given by Ohm's law. Taking into account the equilibrium potential $E_x$, we get 

\begin{align*}
I_x &= g_x (V-E_x)\\
&= \bar{g}_x f_x(n_{x_1},\dots,n_{x_{k_x}})(V-E_x).
\end{align*}
  
In Figure \ref{EquivalentCircuit} below we give the example of a conductance-based model with two types of channels with conductances $g_1$ and $g_2$.

\setlength{\intextsep}{10pt plus 2pt minus 4pt}
\setlength{\floatsep}{10pt plus 2pt minus 4pt}
\setlength{\textfloatsep}{10pt plus 2pt minus 4pt}

\begin{figure}[!ht]
\shorthandoff{:!}
\begin{center}
\includegraphics[scale=0.8]{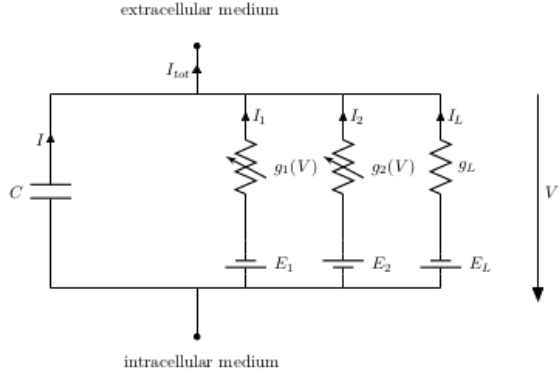}
\end{center}
\shorthandon{:!}
\caption{Equivalent circuit for a conductance-based model with two types of channels}
\setlength{\abovecaptionskip}{2pt}
\setlength{\belowcaptionskip}{15pt}
\label{EquivalentCircuit}
\end{figure}

\setlength{\intextsep}{5pt plus 2pt minus 4pt}
\setlength{\floatsep}{5pt plus 2pt minus 4pt}
\setlength{\textfloatsep}{5pt plus 2pt minus 4pt}

The total current $I_{tot}$ is given by 

\[
I_{tot} = I + I_1 + I_2 + I_L,
\]

where $I = C\frac{\mathrm{d}V}{\mathrm{d}t}$, $I_{1,2}=g_{1,2}(V)(V-E_{1,2})$ and $I_L=g_L(V-E_L)$.

\medskip

The first conductance-based model dates back to the seminal work of Hogkin and Huxley (\cite{HH}) on the squid giant axon. In voltage-clamp experiments (i.e., experiments in which the membrane potential was held fixed), they showed how the ionic currents could be interpreted in terms of changes in $Na^+$ and $K^+$ conductances. From the experimental data, they inferred the dependencies, on the membrane potential and the time, of these conductances. The resulting mathematical model became very popular because it was able to reproduce all key biophysical properties of an action potential. The $K^+$ channels are composed of a single population. Let us denote by $n$ the probability that a channel of type $K^+$ opens. The $K^+$ conductance is given by 

\[
g_K = \bar{g}_Kn^4.
\]

The population of $Na^+$ is composed of two subpopulations and we write $m$ and $h$ the corresponding probabilities that a certain type of gate opens. The $Na^+$ conductance is given by 

\[
g_{Na} = \bar{g}_{Na}m^3h.
\] 

\noindent The total membrane current $I_{tot}$ is then given by

\[
I_{tot} = C\frac{\mathrm{d}V}{\mathrm{d}t} + \bar{g}_Kn^4(V-E_K) + \bar{g}_{Na}m^3h(V-E_{Na}) + g_L(V-E_L),
\]
with $V$ the membrane potential. If an external current $I_{ext}$ is applied to the cell, we can write the dynamic system (HH) for the evolution of the membrane potential 

\begin{equation*}
(HH)\left\{
\begin{aligned}
C \dot{V}(t) &=  \bar{g}_Kn^4(t)(E_K - V(t)) +\bar{g}_{Na}m^3(t)h(t)(E_{Na}-V(t))\\
& \qquad \qquad + g_L(E_L-V(t)) + I_{ext}(t),\\
 \dot{n}(t) &= \alpha_n(V(t))(1-n(t)) - \beta_n(V(t))n(t),\\
\dot{m}(t) &= \alpha_m(V(t))(1-m(t)) - \beta_m(V(t))m(t),\\
 \dot{h}(t) &= \alpha_h(V(t))(1-h(t)) - \beta_h(V(t))h(t).
\end{aligned}
\right.
\end{equation*}

\noindent The expression of the functions $\alpha_x$ and $\beta_x$ and the numerical values of the constants can be found in Appendix \ref{AppendixHH}.

\medskip

To end this section, we give a formal mathematical definition of what we will refer to as a conductance-based model in the sequel.

\begin{definition}{\it Conductance based model.}\\
Let $n\in\N^*$. Let also $k\in\N^*$ and for all $i\in\{1,\dots,k\}$, let $j_i\in\N^*$ such that $\sum_{i=1}^k j_i = n-1$. We call $n$-dimensional conductance-based model the following dynamical system in $\R^n$

\begin{equation*}
\dot{x}_1(t) = \frac{1}{C}\Big(\sum_{i=1}^k \bar{g}_if_i(x_{j_1+\dots+j_{i-1}+1}(t),\dots,x_{j_1+\dots+j_{i-1}+j_i}(t))\big(E_i-x_1(t)\big)\Big),
\end{equation*}
with the convention that $j_1+\dots+j_{i-1}+1 = 2$ and $j_1+\dots+j_{i-1}+j_i = j_1$ for $i=1$, and for $i\in\{2,\dots,n\}$,

\begin{equation*}
\dot{x}_{i}(t) = \alpha_{i}(x_1(t))(1-x_i(t)) - \beta_{i}(x_1(t))x_i(t),
\end{equation*}

where $C>0$ and for all $i\in\{1,\dots,k\}$ and $l\in\{2,\dots,n\}$
\begin{itemize}
\item $\bar{g}_i>0$, $f_i : \R^{j_i}\rightarrow \R_+$ is a smooth function, 
\item $\alpha_l,\beta_l : \R\rightarrow \R$ are smooth functions such that for all $v\in\R, \alpha_l(v)+\beta_l(v)\neq 0$.
\end{itemize}

We finally require that the previous dynamical system exhibits an equilibrium point $x^{\infty}\in\R^n$, that we call resting state, defined by the following equations

\begin{equation*}
x_i^{\infty} = \frac{\alpha_i(x_1^{\infty})}{\alpha_i(x_1^{\infty})+\beta_i(x_1^{\infty})}, \quad \forall i\in\{2,\dots,n\},
\end{equation*}

and 

\begin{equation*}
0=\sum_{i=1}^k \bar{g}_if_i(x_{j_1+\dots+j_{i-1}+1}^{\infty},\dots,x_{j_1+\dots+j_{i-1}+j_i}^{\infty})\big(E_i-x_1^{\infty}\big)
\end{equation*}
	
\end{definition}

Conductance based models are uniquely defined on $\R_+$. The initial conditions $y\in\R^n$ that we consider are physiological conditions with $y_1$ in a physiological range for the membrane potential of the cell considered, basically $y_1\in[V_{min},V_{max}]$ with $-\infty<V_{min}<V_{max}<+\infty$, and $y_i\in[0,1]$ for all $i\in\{2,\dots,n\}$.

\subsection{The Pontryagin Maximum Principle for minimal time single-input affine problems}\label{PMPSection}

In this section we recall the necessary optimality conditions of the Pontryagin Maximum Principle applied to the specific affine problem that we investigate in the sequel.

\noindent Consider the minimum time problem for a smooth single-input affine system: 

\begin{equation}	\label{affineGeneralSystem}
\dot{x}(t) = F_0(x(t))+u(t)F_1(x(t)),\quad x(0) = x_{eq}\in\mathbb{R}^n,
\end{equation}
where $x(t)\in\mathbb{R}^n$ and $x_{eq}$ solution of $F_0(x)=0$ (i.e., an equilibrium point for the uncontrolled system). The control domain $U:=[0,u_{max}]$ is a segment of $\mathbb{R}_+$, with $u_{max}>0$. The state variable must satisfy the final condition $x(t_f)\in M_f$ where 

\begin{equation*}
M_f := \{x\in\mathbb{R}^n | x_1 = V_f\},
\end{equation*}
with $V_f > 0$ a given constant that will later correspond to the potential of a spike. The set of admissible controls, denoted $\mathcal{U}_{ad}$, is the subset of the measurable applications from $\R_+$ to $U$, denoted by $\mathcal{L}(\R_+,U)$, such that (\ref{affineGeneralSystem}) has a unique solution on $\R_+$.

We introduce the Hamiltonian $\mathcal{H} : \mathbb{R}^n\times \mathbb{R}^n\times \mathbb{R}_-\times U \rightarrow \mathbb{R}$ defined for $(x,p,p^0,u)\in \mathbb{R}^n\times \mathbb{R}^n\times \mathbb{R}_-\times U $ by 

\begin{equation}\label{generalHamiltonian}
\mathcal{H}(x,p,p^0,u) := \langle  p,F_0(x)  \rangle + u \langle p,F_1(x) \rangle + p^0,
\end{equation}
where $\langle \cdot,\cdot \rangle$ is the scalar product on $\R^n$, $p\in\R^n$ is the adjoint vector and $p^0\leq 0$ a non-positive number. The Pontryagin Maximum Principle (see \cite{PontryaginBook}, \cite{TrelatBook}) states that if the trajectory $t\rightarrow x^u(t)$, $t\in[0,t_f]$ associated with the admissible control $u \in \mathcal{U}_{ad}$ is optimal on $[0,t_f]$, then there exists $p:[0,t_f]\rightarrow \R^n$ absolutely continuous and $p^0\in\R_-$ such that $(p,p^0)$ is non zero and such that $p$ satisfy the following equations, almost everywhere in $[0,t_f]$:

\begin{align*}
\dot{x}^u(t) &= \frac{\partial \mathcal{H}}{\partial p}(x^u(t),p(t),p^0,u(t)),\\
\dot{p}(t) &= -\frac{\partial \mathcal{H}}{\partial x}(x^u(t),p(t),p^0,u(t)).
\end{align*} 
Moreover, the following maximum condition must be satisfied on $[0,t_f]$: 

\begin{equation}\label{generalMaxCond}
\mathcal{H}(x^u(t),p(t),p^0,u(t)) = \max_{v\in U} \mathcal{H}(x^u(t),p(t),p^0,v).
\end{equation}
In view of the initial and final conditions on the state variable, the transversality condition on $p(0)$ is empty and the one on $p(t_f)$ gives

\begin{align*}
p_1(t_f)  & = \lambda_1 \in \R, \\
p_i(t_f)  & = 0, \quad \forall i \in \{2,\dots,n\}.
\end{align*}

In our particular setting, the augmented system does not depend on the time variable. This implies that the right hand side of (\ref{generalMaxCond}) is constant on $[0,t_f]$. Now since there is no final cost and because the final time is not fixed, we also have 

\begin{equation*}
\max_{v\in U} \mathcal{H}(x^u(t_f),p(t_f),p^0,v) = 0.
\end{equation*}
The two latter remarks imply that for all $t\in[0,t_f]$

\begin{equation}
\mathcal{H}(x^u(t),p(t),p^0,u(t)) = 0 = \max_{v\in U} \mathcal{H}(x^u(t),p(t),p^0,v),
\end{equation}
which can be written, in view of (\ref{generalHamiltonian}): 

\begin{align}
& \langle  p(t),F_0(x^u(t))  \rangle + u(t) \langle p(t),F_1(x^u(t)) \rangle + p^0 = 0\label{NullAffineHamiltonian}\\
& \qquad = \langle  p(t),F_0(x^u(t))  \rangle + \max_{v\in U} v\langle p(t),F_1(x^u(t)) \rangle + p^0.\label{AffineMaxCond}
\end{align}

In the case of single-input affine systems, the maximum condition (\ref{AffineMaxCond}) gives the expression of the optimal control:

\begin{equation*}
u(t):=\left\{
\begin{aligned}
& u_{max}, \quad & \text{if } \langle p(t),F_1(x^u(t)) \rangle > 0,\\
& 0, \quad & \text{if } \langle p(t),F_1(x^u(t)) \rangle < 0,\\
& \text{undetermined}, \quad & \text{if } \langle p(t),F_1(x^u(t)) \rangle = 0. 
\end{aligned}
\right.
\end{equation*}

The function $\varphi(t) := \langle p(t),F_1(x^u(t)) \rangle$, whose sign gives the expression of the optimal control is called the switching function. If it does not vanish on any subinterval $I$ of $[0,t_f]$, the optimal control is a succession of constant controls called bang-bang control. The switching times between the two constant modes are given by the change of sign of the switching function $\varphi$. This conclusion fails if there exists a subinterval $I$ of $[0,t_f]$ along which the switching function vanishes. The control on $I$ is then called singular and this situation has to be further investigated. 

Finally, the non-triviality of $(p,p^0)$ reduces in fact to the one of $p$ because if $p(t)=0$ for a given $t\in[0,t_f]$ then $p^0 = 0$ because of (\ref{NullAffineHamiltonian}).

The investigation of the existence of singular trajectories will be done later for our different models but for now let us state that if there exists a subinterval $I$ on which the switching function vanishes, with $u$ the corresponding control, then from the Pontryagin Maximum Principle, $(x^u,p,u)$ is the solution, on $I$, of the following equations: 

\begin{equation*}
\dot{x}^u(t) = \frac{\partial \mathcal{H}}{\partial p}(x^u(t),p(t),p^0,u(t)), \quad \dot{p}(t) = -\frac{\partial \mathcal{H}}{\partial x}(x^u(t),p(t),p^0,u(t)), \quad \langle p(t),F_1(x^u(t)) \rangle =0.
\end{equation*}

\section{Control of conductance-based models via Optogenetics}\label{resultsSection}
In this section we consider a general conductance-based model in $\mathbb{R}^n$, with $n\in\mathbb{N}^*$, of the form 

\begin{equation}\label{generalCondBasedModel}
\dot{x}(t) = f_0(x(t)),	\quad t\in \mathbb{R}_+,\quad x(0) = x_0 \in \mathcal{D}\subset \mathbb{R}^n,
\end{equation}
with $f_0$ a smooth vector field in $\mathbb{R}^n$ and $\mathcal{D}$ physiological domain.

\medskip

Optogenetics is a recent and innovative technique which allows to induce or prevent electric shocks in living tissue, by means of light stimulation. Succesfully demonstrated in mammalian neurons in 2005 (\cite{millisecondTimescale}), the technique relies on the genetic modification of cells in order for them to express particular ionic channels, called rhodopsins, whose opening and closing are directly triggered by light stimulation. One of these rhodopsins comes from an unicellular flagellate algae, \textit{Chlamydomonas reinhardtii}, and has been baptized Channelrodhopsins-2 (ChR2). It is a cation channel that opens when illuminated with blue light.

Since the field is very young, the mathematical modeling of the phenomenon is quite scarce. Some models have been proposed, based on the study of the photocycles that the channel go through when it absorbs a photon (see \cite{Nikolic} for a 3-states model and \cite{Hegemann} for a 4-states model). In \cite{Nikolic}, the authors study two models for the ChR2 that are able to reproduce the photocurrents generated by the light stimulation of the channel. Those models are constituted by several states that can be either conductive (the channel is open) or non-conductive (the channel is closed). Transitions between those states are spontaneous, depend on the membrane potential or are triggered by the absorption of a photon. This kind of models has already been used to simulate photocurrents in cardiac cells. In \cite{Abilez}, the authors include ChR2 photocurrents into an infinite dimensional model and use finite differences and elements to simulate the system. The optimal control of such a system is not investigated in this paper.
Here we are interested in both 3-states and 4-states models of Nikolic and al. \cite{Nikolic}. The 3-states model has one open state $o$ and two closed states $c$ and $d$ while the 4-states model has two open states $o_1$ and $o_2$, and two closed states $c_1$ and $c_2$. Their transitions are represented on Figures \ref{ChR2_3States} and \ref{ChR2_4States}.

\begin{figure}[!ht]
\begin{center}
\begin{tikzpicture}[->,>=stealth',shorten >=1pt,auto,node distance=3cm,
  thick,main node/.style={circle,draw,font=\sffamily\Large\bfseries}]

  \node[main node] (1) {c};
  \node[main node] (2) [right of=1] {d};
  \node[main node] (3) [below left of=2] {o};

  \path[every node/.style={font=\sffamily\small}]
    (1) edge [bend right] node [left] {$u(t)$} (3)
    (2) edge  [bend right] node [above]  {$K_r$} (1)
    (3) edge [bend right] node [right] {$K_d$} (2);

\end{tikzpicture}
\end{center}
\caption{ChR2 three states model}
\label{ChR2_3States}
\end{figure}
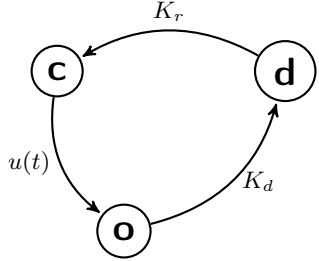

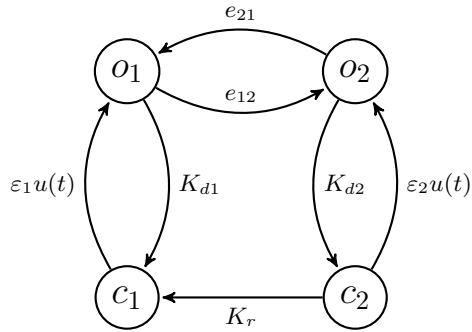
\begin{figure}[!ht]
\begin{center}
\begin{tikzpicture}[->,>=stealth',shorten >=1pt,auto,node distance=3cm,
  thick,main node/.style={circle,draw,font=\sffamily\Large\bfseries}]

  \node[main node] (1) {$o_1$};
  \node[main node] (2) [right of=1] {$o_2$};
  \node[main node] (3) [below of=2] {$c_2$};
  \node[main node] (4) [below of=1] {$c_1$};

  \path[every node/.style={font=\sffamily\small}]
    (1) edge [bend left] node [right] {$K_{d1}$} (4)
        edge [bend right] node [above] {$e_{12}$} (2)
    (2) edge [bend right] node [above] {$e_{21}$} (1)
        edge [bend right] node[right] {$K_{d2}$} (3)
    (3) edge [bend right] node [right] {$\varepsilon_2 u(t)$} (2)
        edge node {$K_r$} (4)
    (4)edge [bend left] node[left] {$\varepsilon_1 u(t)$} (1);
        
\end{tikzpicture}
\setlength{\abovecaptionskip}{10pt}
\caption{ChR2 four states model. }\label{ChR2_4States}
\end{center}
\end{figure}

In the 3-states model, the transition from the dark adapted close state $c$ and the open state $o$ is controlled by a function $u(t)$, proportional to the intensity of the light applied to the neuron. In our model, the intensity is then the control variable. The transition from the open state to the light adapted close state $d$ is spontaneous and has a time constant very small in front of the one of the transition from $d$ to $c$ (i.e. $1/K_d << 1/K_r$). This last transition represents the fact that the protein has to regenerate before being able to go through a new cycle. The 4-states model can be similarly interpreted. The transitions from closed states to open states are triggered by light stimulation and all the other transitions are independent of the intensity of the light applied to the neuron. Hence, $\varepsilon_1$, $\varepsilon_2$, $e_{12}$, $e_{21}$, $K_{d1}$, $K_{d2}$ and $K_r$ are all positive constants.  This constitutes our general assumption on the models we study. Indeed, we assume that the transitions from closed states to open states depend linearly on the light and that all the others are independent of the light. This assumption is not too heavy since it leads to models that  still reproduces the shape of the photocurrents produced by the channel, and experimentally measured. Furthermore, it makes our control system affine. The dynamical system based on Figures \ref{ChR2_3States} and \ref{ChR2_4States} is given by 

\begin{equation}\label{ChR2_3States_Dynamics}
\left\{
\begin{aligned}
\dot{o}(t) & = u(t)(1-o(t)-d(t)) - K_do(t)\\
\dot{d}(t) & = K_do(t) - K_rd(t),
\end{aligned}
\right.
\end{equation}

and 

\begin{equation}\label{ChR2_4States_Dynamics}
\left\{
\begin{aligned}
\dot{o}_1(t) & =  \varepsilon_1u(t)(1-o_1(t)-o_2(t)-c_2(t)) - (K_{d1} + e_{12})o_1(t) + e_{21}o_2(t),\\
\dot{o}_2(t) & =  \varepsilon_2u(t)c_2(t) + e_{12}o_1(t) - (K_{d2} + e_{21})o_2(t),\\
\dot{c}_2(t) & =  K_{d2}o_2(t) - (\varepsilon_2u(t)+K_r)c_2(t).
\end{aligned}
\right.
\end{equation}

In the 3-states model, the conductance of the ChR2 channel is assumed to be proportional to the probability $o(t)$ that the channel opens, so that the ion current associated to ChR2 channels is given by

\[
I_{ChR2}(t) = g_{ChR2}o(t)(V_{ChR2}-v(t)),
\]
with $v$ the membrane potential of the channel, $g_{ChR2}$ the maximal conductance of the channel and $V_{ChR2}$ the equilibrium potential of the channel. See Appendix \ref{NumericalConstantsChR2} for the numerical computation of these constants.
In the 4-states model, the open states are assumed to be of different conductivity so that 

\[
I_{ChR2}(t) = g_{ChR2}(o_1(t)+\rho o_2(t))(V_{ChR2}-v(t)),
\]

with $\rho\in (0,1)$.
We can now include these two models of ChR2 in a conductance-based model defined in the previous section.

\begin{definition}\label{DefControlledConductanceModel}

\begin{itemize}
\item[i)] We call ChR2-3-states controlled conductance-based model, the system given by 

\begin{equation}\label{3StatesControlledConductanceBasedModel}
\left\{
\begin{aligned}
\dot{x}(t) & = f_0(x(t)) + \frac{1}{C}g_{ChR2}o(t)(V_{ChR2}-x_1(t))\mathbf{e}_1\\
\dot{o}(t) & = u(t)(1-o(t)-d(t)) - K_do(t)\\
\dot{d}(t) & = K_do(t) - K_rd(t),
\end{aligned}
\right.
\end{equation}
with $\mathbf{e}_1=(1,0,\dots,0)\in \R^n$. We rewrite this system in $\mathbb{R}^{n+2}$ in the affine form 

\begin{equation}
\dot{y}(t) = \tilde{f}_0(y(t)) + u(t)f_1(y(t)), \quad t\in\mathbb{R}_+, 	
\end{equation}
with $y(\cdot)=(x(\cdot),o(\cdot),d(\cdot))$, $\tilde{f}_0(y) = (f_0(x)+ \frac{1}{C}g_{ChR2}o(t)(V_{ChR2}-x_1(t))\mathbf{e}_1,-K_do,K_do-K_rd)$ and $f_1(y) = (1-o-d)\partial_o$, where $\partial_o$ is the derivative with respect to the variable $o$.

\item[ii)] We call ChR2-4-states controlled conductance-based model, the system given by 

\begin{equation}\label{4StatesControlledConductanceBasedModel}
\left\{
\begin{aligned}
\dot{x}(t) & = f_0(x(t)) + \frac{1}{C}g_{ChR2}(o_1(t)+\rho o_2(t))(V_{ChR2}-x_1(t))\mathbf{e}_1\\
\dot{o}_1(t) & =  \varepsilon_1u(t)(1-o_1(t)-o_2(t)-c_2(t)) - (K_{d1} + e_{12})o_1(t) + e_{21}o_2(t),\\
\dot{o}_2(t) & =  \varepsilon_2u(t)c_2(t) + e_{12}o_1(t) - (K_{d2} + e_{21})o_2(t),\\
\dot{c}_2(t) & =  K_{d2}o_2(t) - (\varepsilon_2u(t)+K_r)c_2(t).
\end{aligned}
\right.
\end{equation}
We also rewrite the system in $\R^{n+3}$,

\begin{equation}
\dot{z}(t) = \hat{f}_0(z(t)) + u(t)f_2(z(t)), \quad t\in\mathbb{R}_+, 	
\end{equation}

with $z(\cdot)=(x(\cdot),o_1(\cdot),o_2(\cdot),c_2(\cdot))$,

\begin{align*}
\hat{f}_0(z) &= (f_0(x)+ \frac{1}{C}g_{ChR2}(o_1(t)+\rho o_2(t))(V_{ChR2}-x_1(t))\mathbf{e}_1,\\
& \qquad \qquad - (K_{d1} + e_{12})o_1 + e_{21}o_2,e_{12}o_1 - (K_{d2} + e_{21})o_2, K_{d2}o_2),
\end{align*}
and 
\[
f_2(z) = \varepsilon_1(1-o_1-o_2-c_2)\partial_{o_1} + \varepsilon_2c_2\partial_{o_2}- \varepsilon_2c_2\partial_{c_2}.
\]

\end{itemize}
\end{definition}

\begin{notation}
Let $k\in\N^*$. We use two ways to write a vector field $F:\R^k\rightarrow\R^k$. For $x\in\R^k$, we write either
\begin{itemize}
\item $F(x) = (F_1(x),\dots,F_k(x))$, or
\item $F(x) = F_1(x)\partial_1 + \cdots + F_k(x)\partial_k$,
\end{itemize}  
where $F_i:\R^k\rightarrow \R$ is the $i^{\mathrm{th}}$ coordinate of $F$ and $\partial_i$ is the partial derivative along the $i^{\mathrm{th}}$ direction, for $i\in\{1,\dots,k\}$.
\end{notation}

We already used this mixed notation in Definition \ref{DefControlledConductanceModel} above. The second notation will be useful for the computation of Lie brackets later in this paper.

Note that for a bounded measurable function $u:\R_+\rightarrow \R$ and a starting point $((o_0,d_0),(o_1,o_2,c_2))\in\R^2\times\R^3$, the systems (\ref{ChR2_3States_Dynamics}) and (\ref{ChR2_4States_Dynamics}) admit a unique solution, absolutely continuous  on $\R_+$. Thus, for all bounded measurable function $u:\R_+\rightarrow \R$ and all initial conditions $y_0\in\mathcal{D}\times\R^2$ and $z_0\in\mathcal{D}\times\R^3$, the systems (\ref{3StatesControlledConductanceBasedModel}) and (\ref{4StatesControlledConductanceBasedModel}) have a unique solution, defined on $\R_+$ and such that $x(\cdot)$ is of class $C^1$ and $(o(\cdot),d(\cdot))$ and $(o_1(\cdot),o_2(\cdot),c_2(\cdot))$ are absolutely continuous on $\R_+$.

\subsection{The minimal time spiking problem}

The control problem we are interested in here can be formulated for both ChR2 models. Consider a conductance-based neuron model in its resting state. If no light is applied to the neuron (i.e. $u\equiv 0$) then the system stays in this resting state. We want to find the optimal control that triggers a spike in minimum time when starting from the resting state. To do so, let $V_s > 0$ be the membrane potential that we decide to be corresponding to a spike. Since the control is proportional to the intensity of the light applied to the neuron, the control space $U$ will be a segment $[0,u_{max}]$, with $u_{max}>0$. Let $x_{eq}\in\R^n$ a resting state of the conductance-based model. In the next two sections, we formulate the mathematical problem for both ChR2 models.

\subsubsection{The ChR2 3-states model}

Let $y_0 = (x_{eq},0,0) \in\R^{n+2}$ be our starting point. The state $(0,0)$ for the system (\ref{ChR2_3States_Dynamics}) corresponds to a neuron being in the dark for quite a long period of time (i.e. all the ChR2 channels are in the dark adapted close state $c$). From $y_0$, we then want to reach in minimal time (denoted $t_f$) the manifold  

\begin{equation*}
M_s := \{y\in\mathbb{R}^{n+2} | y_1 = V_s\}.
\end{equation*}

As in Section \ref{PMPSection} we define $\mathcal{H} :\mathbb{R}^{n+2} \times \mathbb{R}^{n+2}\times \mathbb{R}_-\times U \rightarrow \mathbb{R}$ the Hamiltonian of the system for $(y,p,p^0,u)\in \mathbb{R}^{n+2}\times \mathbb{R}^{n+2}\times \mathbb{R}_-\times U$ by 

\begin{equation}\label{3StatesExtendedGeneralHamiltonian}
\mathcal{H}(y,p,p^0,u) := \langle  p,\tilde{f}_0(y)  \rangle + u \langle p,f_1(y) \rangle + p^0.
\end{equation}

This control problem falls into the framework of Section \ref{PMPSection}. If there is no singular extremal, the optimal control is bang-bang and is given by the sign of the switching function. Let $p : \R_+ \rightarrow \R^{n+2}$ be the adjoint vector of the Pontryagin Maximum Principle. The switching function reads, for $t\in [0,t_f]$,

\[
\varphi(t):=(1-o(t)-d(t))p_o(t) \text{  or also }(1-y_{n+1}(t)-y_{n+2}(t))p_{n+1}(t).
\]

In the absence of singular extremals, if we write $u^* : [0,t_f]\rightarrow U$ the optimal control, then 

\[
u^*(t) = u_{max}\mathbf{1}_{\varphi(t) > 0}, \quad \forall t\in[0,t_f].
\]

\subsubsection{The ChR2 4-states model}

We define here the same quantities for the 4-states model.
Let $z_0 = (x_{eq},0,0,0) \in\R^{n+3}$ be our starting point. From $z_0$, we then want to reach in minimal time (denoted $t_f$) the manifold  

\begin{equation*}
M_s := \{z\in\mathbb{R}^{n+3} | y_1 = V_s\}.
\end{equation*}
The Hamiltonian $\mathcal{H} :\mathbb{R}^{n+3} \times \mathbb{R}^{n+3}\times \mathbb{R}_-\times U \rightarrow \mathbb{R}$ is defined for $(z,q,q^0,u)\in \mathbb{R}^{n+3}\times \mathbb{R}^{n+3}\times \mathbb{R}_-\times U$ by 

\begin{equation}\label{4StatesExtendedGeneralHamiltonian}
\mathcal{H}(y,q,q^0,u) := \langle  q,\hat{f}_0(z)  \rangle + u \langle q,f_2(z) \rangle + q^0.
\end{equation}
Let $q : \R_+ \rightarrow \R^{n+2}$ be the adjoint vector of the Pontryagin Maximum Principle. The switching function writes, for $t\in [0,t_f]$,

\begin{align*}
\psi(t):=\varepsilon_1(1-o_1(t)-o_2(t)-c_2(t))q_{o_1}(t) + \varepsilon_2c_2(t)q_{o_2}(t)- \varepsilon_2c_2(t)q_{c_2}(t).
\end{align*}
Singular extremals correspond to vanishing switching functions. We will treat the two ChR2 models in a different way. Indeed, the 3-states model is theoretically tractable and is the object of the following section. The 4-states will be investigated numerically.

\subsection{The Goh transformation for the ChR2 3-states model}\label{GohSection}

We state and prove here our main reduction result regarding the existence of optimal singular controls for the ChR2-3-states control problem. 

\begin{theorem}\label{ReductionTheorem}
The existence of optimal singular extremals in the spiking problem in minimal time for the control system (\ref{3StatesControlledConductanceBasedModel}) is equivalent to the existence of optimal singular extremals in the same problem but for the reduced system on $\R^n$
\[
\dot{x} = f_0(x) + o\tilde{f_1}(x),
\]
where $o$ is the control variable and $\tilde{f_1}(x) =\frac{1}{C}g_{ChR2}(V_{ChR2}-x_1)\mathbf{e}_1 $. 

\end{theorem}

Every nonlinear control system of the form $\dot{x}=f(x,u)$ can be interpreted as an affine one by making the transformation $\dot{u}=v$ and considering the variable $v$ as the new control and the variable $(x,u)$ as the new state variable. The inverse transformation, called the Goh transformation, is a great tool for the investigation of singular extremals and will reveal itself fundamental here to show the absence of optimal singular trajectories in the models we will consider later. 

\begin{notations}	
To every couple of points $y:=(x,o,d)\in\R^{n+2}$ and $p:=(p_x,p_o,p_d)\in\R^{n+2}$ we associate a couple of points of $\R^{n+1}$ defined by $\tilde{y}:=(x,d)$ and $\tilde{p}:=(p_x,p_d)$. Moreover, we write the corresponding reduced Hamiltonian $\tilde{\mathcal{H}}$ defined for $(\tilde{y},\tilde{p},p^0)\in\R^{n+1}\times \R^{n+1}\times \R_-$ and $o\in\R$ by $\tilde{\mathcal{H}}(\tilde{y},\tilde{p},p^0,o):= \langle \tilde{p},\tilde{f}_0(\tilde{y})\rangle + o \langle \tilde{p},\tilde{f}_1(\tilde{y})\rangle + p^0$, where the vector field $\tilde{f}_0$ remains unchanged (it did not depend on the variable $o$) and the vector field $\tilde{f}_1$ is defined, for all $\tilde{y}\in\R^{n+1}$, by $\tilde{f}_1(\tilde{y}) := g_{ChR2}(V_{ChR2}-\tilde{y}_1)\partial_1$.
\end{notations}

The following lemma is the first step to reduce the dimension of the system that has to be considered to investigate the existence of singular extremals.

\begin{lemma}\label{GohLemma}

$(y,p)$ is the projection, on the space of continuous functions from $\R_+$ to $\R^{n+2}\times\R^{n+2}$,  of a solution $(y,p,u)$ of

\begin{equation}\label{LemmaSingularSystem}
\dot{y}(t) = \frac{\partial \mathcal{H}}{\partial p}(y(t),p(t),p^0,u(t)), \quad \dot{p}(t) = -\frac{\partial \mathcal{H}}{\partial y}(y(t),p(t),p^0,u(t)), \quad \langle p(t),f_1(y(t)) \rangle =0.
\end{equation} 
if and only if $p_o\equiv 0$ and $(\tilde{y},\tilde{p})$ is a solution of 

\begin{equation}\label{LemmaSingularReducedSystem}
\dot{\tilde{y}}(t) = \frac{\partial \tilde{\mathcal{H}}}{\partial \tilde{p}}(\tilde{y}(t),\tilde{p}(t),p^0,o(t)), \quad \dot{\tilde{p}}(t) = -\frac{\partial \tilde{\mathcal{H}}}{\partial \tilde{y}}(\tilde{y}(t),\tilde{p}(t),p^0,o(t)), \quad \langle \tilde{p}(t),\tilde{f}_1(\tilde{y}(t)) \rangle =0.
\end{equation}

%%%%%%%%%%%%%%%%%%%%%%%%%%%%%%%%%%%%%%%%%%%%%%%%%%%%%%%%%%%%
\begin{comment}
Furthermore, we have the following two identities: 

\begin{itemize}
\item[i.] \[\left(\frac{\partial}{\partial t}\frac{\partial \mathcal{H}}{\partial u}\right)_{|(y,p,u)} = -\frac{\partial \tilde{\mathcal{H}}}{\partial o}_{|(\tilde{y},\tilde{p},o)}.\]
\item[ii.] \[\left(\frac{\partial}{\partial u}\frac{\partial^2}{\partial t^2}\frac{\partial \mathcal{H}}{\partial u}\right)_{|(y,p,u)} = -\frac{\partial^2 \tilde{\mathcal{H}}}{\partial o^2}_{|(\tilde{y},\tilde{p},o)}.\]
\end{itemize}
\end{comment}
%%%%%%%%%%%%%%%%%%%%%%%%%%%%%%%%%%%%%%%%%%%%%%%%%%%%%%%%%%%%%

\end{lemma}

This lemma shows that singular extremals of (\ref{3StatesControlledConductanceBasedModel}) are directly related to singular extremals of the following, and still affine control system: 

\begin{equation}\label{GohControlledConductanceBasedModel}
\left\{
\begin{aligned}
\dot{x}(t) & = f_0(x(t)) + g_{ChR2}o(t)(V_{ChR2}-x_1(t))\mathbf{e}_1,\\
\dot{d}(t) & = K_do(t) - K_rd(t),
\end{aligned}
\right.
\end{equation}
where the control is now the variable $o$.

%%%%%%%%%%%%%%%%%%%%%%%%%%%%%%%%%%%%%%%%%%%%%%%%%%%%%%%%%%%%
\begin{comment}
Since the reduced control system remains affine, the identity $ii.$ of Lemma \ref{GohLemma} shows that, for the initial system, the Legendre-Clebsch condition is empty. Of course, and for both system, the second-order Legendre condition is also empty.
\end{comment}
%%%%%%%%%%%%%%%%%%%%%%%%%%%%%%%%%%%%%%%%%%%%%%%%%%%%%%%%%%%%%

In the models that we are going to study in the sequel, we will see that this transformation allows to conclude to the absence of optimal singular extremals.

Before going through the proof of Lemma \ref{GohLemma}, let us here introduce the notion of Lie brackets for regular vector fields. We give two equivalent definitions, depending on the notation used for the vector fields.

Let $k\in\N^*$ and $g,h : \R^k\rightarrow \R^k$ two vector fields of class $C^1$. Let $(g_1,\dots,g_k)$ and $(h_1,\dots,h_k)$ their coordinate mappings. The Lie bracket $[g,h]: \R^k\rightarrow \R^k$ of $g$ and $h$ is the vector field defined for $x\in\R^k$ by 

\[
[g,h](x) = J_h(x)g(x) - J_g(x)h(x),
\]
or equivalently by 

\[
[g,h](x) = \sum_{i=1}^k\sum_{j=1}^k \big(g_j(x)\partial_jh_i(x) - h_j(x)\partial_jg_i(x)\big)\partial_i,
\]
where $J_h$ and $J_g$ are the Jacobian matrices of $h$ and $g$. The expression $J_h(x)g(x)$ has to be understood as the product of the $k\times k$-matrix by the $k$-vector. Further in this paper we will use the convenient notation $$\mathrm{ad}_h g := [h,g]$$ that allows to reduce expressions of multiple Lie brackets. Finally, one important relation for the computation of singular controls is the following. Let $(x^u,p)$ be an extremal pair of the Pontryagin maximum principle associated to a control $u$. Then for any smooth vector field $h:\R^k\rightarrow \R^k$ and all $t\in[0,t_f]$,

\begin{equation}\label{derivative}
\frac{\mathrm{d}}{\mathrm{d}t} \langle p(t), h(x^u(t))\rangle  = \langle p(t), [F_0,h](x^{u}(t))\rangle + u(t)\langle p(t), [F_1,h](x^{u}(t))\rangle. 
\end{equation}
\begin{proof}{ of Lemma \ref{GohLemma}.} The proof comes from the general result of Section 1.9.4 of \cite{BonnardKupka} and the shape of our particular model. If we keep on writing $y=(x,o,d)$, system (\ref{LemmaSingularSystem}) gives on an interval $I$ of $[0,t_f]$:

\begin{equation}\label{DevelopedSystem}
\left\{
\begin{aligned}
\dot{x} &= f_0(x) + g_{ChR2}o(V_{ChR2}-x_1)\mathbf{e}_1,\\
\dot{d} &= k_do-k_rd,\\
\dot{o}&=(1-o-d)u-K_do,\\
\dot{p}_x &= -J_{f_0}^tp_x + g_{ChR2}p_o\mathbf{e}_1,\\
\dot{p}_d &= up_o +K_rp_d,\\
\dot{p}_o & = -g_{ChR2}(V_{ChR2}-x_1)p_{x_1} -K_dp_d+(u+K_d)p_o,\\
0&=(1-o-d)p_o,
\end{aligned}
\right.
\end{equation}
where $J_{f_0}^t$ is the transpose of the Jacobian matrix of $\tilde{f}_0$.
For continuity reasons, we get that either $p_o\equiv0$ or $(1-o-d)\equiv 0$ on $I$. If $(1-o-d)\equiv 0$ then $-K_rd = \dot{o}+\dot{d}\equiv 0$ so that $d\equiv 0$ and $o\equiv 1$. But $d\equiv 0 \Rightarrow \dot{d}\equiv 0$ so that $\dot{o}\equiv 0$ which is incompatible with $o\equiv 1$, since $\dot{o}=-K_do$. We conclude that, necessarily, $p_o\equiv0$ on $I$. This equality implies that $\dot{p}_o\equiv 0$ and from the penultimate equation of (\ref{DevelopedSystem}) we get $-g_{ChR2}(V_{ChR2}-x_1)p_{x_1} -K_dp_d\equiv 0$ which also writes $\langle \tilde{p},\tilde{f}_1(\tilde{y}) \rangle \equiv 0$. Now the first two equations of (\ref{DevelopedSystem}) correspond to 
\[
\dot{\tilde{y}}(t) = \frac{\partial \tilde{\mathcal{H}}}{\partial \tilde{p}}(\tilde{y}(t),\tilde{p}(t),p^0,o(t)),
\]
and the $4^{\mathrm{th}}$ and $5^{\mathrm{th}}$ equations correspond to 

\[
\dot{\tilde{p}}(t) = -\frac{\partial \tilde{\mathcal{H}}}{\partial \tilde{y}}(\tilde{y}(t),\tilde{p}(t),p^0,o(t)).
\]
We just showed that (\ref{LemmaSingularSystem}) $\Rightarrow$ ($p_o\equiv 0$ and (\ref{LemmaSingularReducedSystem})). The inverse implication is straightforward.

%%%%%%%%%%%%%%%%%%%%%%%%%%%%%%%%%%%%%%%%%%%%%%%%%%%%%%%%%%%%%
\begin{comment}
\[
\left(\frac{\partial}{\partial t}\frac{\partial \mathcal{H}}{\partial u}\right)_{|(y,p,u)} = \frac{\mathrm{d}}{\mathrm{d} t} \langle p(t),f_1(y(t)) \rangle = \lang p(t), [\tilde{f}_0,f_1](y(t))\rangle,
\]

Identities $i.$ and $ii.$ are trivial when we recall the relation between the successive time derivatives of the switching function and the Lie brackets of the vector fields defining the control system : 

\[
\dot{\phi}(t) = \lang p(t), [f_1,\tilde{f}_0](y(t))\rangle, 
\]
 
 and

 This relat for every vector field $h:\R^{n+2}\rightarrow\R^{n+2}$
\end{comment}
%%%%%%%%%%%%%%%%%%%%%%%%%%%%%%%%%%%%%%%%%%%%%%%%%%%%%%%%%%%%%

\end{proof}

\begin{proof}[Proof of Theorem \ref{ReductionTheorem}]
The result of Lemma \ref{GohLemma} is the first step of the proof. To finish up with it, consider the spiking problem in minimum time for the reduced system (\ref{GohControlledConductanceBasedModel}) : 
\begin{equation*}
\left\{
\begin{aligned}
\dot{x}(t) & = f_0(x(t)) + g_{ChR2}o(t)(V_{ChR2}-x_1(t))\mathbf{e}_1,\\
\dot{d}(t) & = K_do(t) - K_rd(t),
\end{aligned}
\right.
\end{equation*}

Remark that the dynamics of the variables $x$ and $d$ are completely decoupled. Furthermore, the targeted manifold is only defined by the location of variable $x_1$. These two remarks imply that an optimal control for system (\ref{GohControlledConductanceBasedModel}) has to be optimal for the even more reduced control system :

\begin{equation*}
\dot{x}(t) = f_0(x(t)) + g_{ChR2}o(t)(V_{ChR2}-x_1(t))\mathbf{e}_1.
\end{equation*}
\end{proof}

\subsection{Lie bracket configurations for the ChR2 4-states model}

In the case of the ChR2 4-states model, we will observe numerically that the optimal control is bang-bang for various values of the maximum intensity $u_{max}$. Here we give the expression of the first Lie brackets.
In most cases, a singular optimal control $\bar{u}$ would have the expression

\[
\bar{u}(t) = \frac{\langle q(t) , \mathrm{ad}^2_{\hat{f}_0}f_2(z(t))\rangle}{\langle q(t) , \mathrm{ad}^2_{f_2}\hat{f}_0(z(t))\rangle}.
\]
Indeed, if $I$ is an interval of $[0,t_f]$ on with the switching function $\psi$ vanishes, then for $t\in I$,

\begin{align*}
\psi(t) &= 0,\\
\dot{\psi}(t) & =  \langle q(t), [\hat{f}_0,f_2](z(t)) \rangle = 0 ,\\
\ddot{\psi}(t) & = \langle q(t) , \mathrm{ad}^2_{\hat{f}_0}f_2(z(t))\rangle -\bar{u}(t)\langle q(t) , \mathrm{ad}^2_{f_2}\hat{f}_0(z(t))\rangle = 0 ,	
\end{align*}

The expressions of $[\hat{f}_0,f_2]$ and $\mathrm{ad}^2_{f_2}\hat{f}_0$ are not too much complicated since theses brackets have non zero components only on the directions $z_1$, $z_{n+1}$, $z_{n+2}$ and $z_{n+3}$ (independently of $n\in\N^*$), which we also write $v$, $o_1$, $o_2$ and $c_2$. We will not give the expression of $\mathrm{ad}^2_{\hat{f}_0}f_2$ because it is too long and of small interest since we will treat the problem numerically. Let us just mention that it has non zero components on all the directions of the state space $\R^{n+3}$.

\begin{align*}
[\hat{f}_0,f_2](z) & = -\Big(\varepsilon_1(1-o_1-o_2-c_2)+\varepsilon_2\rho c_2\Big)\frac{1}{C}g_{ChR2}(V_{ChR2}-v)\partial_v\\
& \quad +\Big(\varepsilon_1(1-o_1-o_2-c_2)(e_{12}+K_{d1})+\varepsilon_1K_{d1}o_1-(\varepsilon_1K_r-\varepsilon_2e_{21})c_2\Big)\partial_{o_1}\\
& \quad + \Big(-\varepsilon_1(1-o_1-o_2-c_2)e_{12}+\varepsilon_2K_{d2}o_2+ \varepsilon_2(e_{21}+K_{d2}-K_r)c_2\Big)\partial_{o_2}\\
& \quad - \varepsilon_2K_{d2}(o_2+c_2)\partial_{c_2},
\end{align*}

\noindent and

\begin{align*}
\mathrm{ad}^2_{f_2}\hat{f}_0(z) & = -\Big((\varepsilon_1)^2(1-o_1-o_2-c_2)+(\varepsilon_2)^2\rho c_2\Big)\frac{1}{C}g_{ChR2}(V_{ChR2}-v)\partial_v\\
& \quad -\varepsilon_1\Big(\varepsilon_1(1-o_1-o_2-c_2)(e_{12}+K_{d1})+\varepsilon_1K_{d1}o_1-(\varepsilon_1K_r-\varepsilon_2e_{21})c_2\Big)\partial_{o_1}\\
& \quad - \Big((\varepsilon_1)^2(1-o_1-o_2-c_2)e_{12}+(\varepsilon_2)^2K_{d2}o_2+ (\varepsilon_2)^2(-e_{21}+K_{d2}-K_r)c_2\Big)\partial_{o_2}\\
& \quad + (\varepsilon_2)^2K_{d2}(o_2+c_2)\partial_{c_2}.
\end{align*}

\section{Application to some neuron models with numerical results}\label{SectionApplication}
\markboth{Application and numerical results}{}

In this section, we apply the reduction results of Section \ref{GohSection} to some widely used models and support our theoretical results with numerical results. These theoretical results regard the ChR2-3-states model and we also investigate numerically the associated ChR2-4-states models. The numerical results are obtained by direct methods based on the \verb?ipopt? routine \cite{Ipopt} to solve nonlinear optimization problems, and implemented with the \verb?ampl? language \cite{AmplBook}. For a survey on numerical methods in optimal control, see \cite{TrelatAerospace}. The numerical values used for the ChR2-3-states and 4-states models are those of Appendices \ref{AppendixChR2_3States} and \ref{AppendixChR2_4States}. For each neuron model that we study, namely the FitzHugh-Nagumo model, the Morris-Lecar model and the reduced and complete Hodgkin-Huxley models, we implement the direct method for the ChR2-3-states and 4-states models and compare them. We repeat the computation for several values of the maximum control value in order to try and detect possible singular optimal controls. Indeed, it would be possible that a singular optimal control only appears above some threshold of the maximal control value. Nevertheless, no model numerically displays such controls. We then compare the neuron models between them in terms of their behavior with respect to optogenetic control. Physiologically, Channelrhdopsin has a depolarizing effect on a neuron membrane so that it is physiologically intuitive to expect that we need to switch on the light to obtain a spike, and the more light we put in the system, the faster the spike will occur.
We propose to distinguish between two classes of models. The first class comprises neuron models that display the intuitive physiological response to optogenetic stimulation and the second class comprises neuron models that display an unexpected response.

\subsection{The FitzHugh-Nagumo model}

The FitzHugh-Nagumo model is not exactly a conductance-based model but a two-dimensional simplification of the Hodgkin-Huxley model. This model takes his name from the initial work of FitzHugh \cite{FitzHugh} who suggested the system and Nagumo \cite{Nagumo} who gave the equivalent circuit. The idea was to find a simpler model that still featured the mathematical properties of excitation and propagation.

\paragraph{The ChR2-3-states model}\hspace*{0pt}\\

\noindent The $ChR2$-3-states controlled FitzHugh-Nagumo model is

\begin{equation*}
(FHN)\left\{
\begin{aligned}
\dot{v}(t) & =  v(t)-\frac{1}{3}v^3(t)-w(t) +\frac{1}{C}g_{ChR2}o(t)(V_{ChR2}-v(t)),\\ 
\dot{w}(t) & = c(v(t) + a -bw(t)),\\
\dot{o}(t) & = u(t)(1-o(t)-d(t)) - K_do(t),\\
\dot{d}(t) & = K_do(t) - K_rd(t),
\end{aligned}
\right.
\end{equation*}
where $v$ is the membrane potential and $w$ a conductance-like variable that provides a negative feedback, and $a$, $b$ and $c$ are constants. In the original model, the numerical values of these constants were $a=0.7$, $b=0.8$ and $c=0.08$. The adjoint equations write 

\begin{equation*}
(FHN_{adj})\left\{
\begin{aligned}
\dot{p}_v(t) & =  -p_v(t)(1-v^2(t)-\frac{1}{C}g_{ChR2}o(t))-cp_w(t),\\ 
\dot{p}_w(t) & = p_v(t) + bcp_w(t),\\
\dot{p}_o(t) & = -\frac{1}{C}g_{ChR2}(V_{ChR2}-v(t))p_v(t) + (u(t)+K_d)p_o(t) -K_dp_d(t),\\
\dot{p}_d(t) &= u(t)p_o(t) + K_rp_d(t),
\end{aligned}
\right.
\end{equation*}
and the switching function is $\varphi(t) = (1-o(t)-d(t))p_o(t)$. The following lemma gives the optimal control for the minimal time control of the ChR2-controlled FitzHugh-Nagumo model.

\begin{proposition}\label{PropFHN}
The optimal control $u^*:\R_+\rightarrow U$ for the minimal time control of the FitzHugh-Nagumo model is bang-bang and given by 

\begin{equation*}
u^*(t) = u_{max} \mathbf{1}_{p_o(t)>0}, \quad \forall t\in[0,t_f].
\end{equation*}
Furthermore, the optimal control begins with a bang arc of maximal value, i.e.

\[
\exists t_1\in [0,t_f], u^*(t)=u_{max}, \forall t\in[0,t_1]. 
\]

\end{proposition}

\begin{proof}
Let us show that there is no optimal singular extremals. The results for conductance-based models given in section \ref{GohSection} are straightforwardly applicable to the FitzHugh-Nagumo model and the reduced control system is the following

\begin{equation*}
(FHN')\left\{
\begin{aligned}
\dot{v}(t) & =  v(t)-\frac{1}{3}v^3(t)-w(t) +\frac{1}{C}g_{ChR2}u(t)(V_{ChR2}-v(t))\\ 
\dot{w}(t) & = c(v(t) + a -bw(t))\\
\end{aligned}
\right.
\end{equation*}
The adjoint equations for this system are 

\begin{equation*}
(FHN'_{adj})\left\{
\begin{aligned}
\dot{p}_v(t) & = -p_v(t)(1-v^2(t)-\frac{1}{C}g_{ChR2}u(t))-cp_w(t)\\ 
\dot{p}_w(t) & = p_v(t) + bcp_w(t)\\
\end{aligned}
\right.
\end{equation*}
The vector fields defining the affine system ($FHN'$) are

\begin{equation*}
\begin{aligned}
f_0(v,w) & =  (v-\frac{1}{3}v^3-w)\partial_v + c(v + a -bw)\partial_w\\ 
f_1(v,w) & = \frac{1}{C}g_{ChR2}(V_{ChR2}-v)\partial_v\\
\end{aligned}
\end{equation*}
For the reduced system, the switching function is given by 

\begin{equation*}
\phi(t) = 	\langle p(t), f_1(v(t),w(t))\rangle = \frac{1}{C}g_{ChR2}(V_{ChR2}-v(t))p_v(t).
\end{equation*}

\paragraph{Investigation of singular trajectories}\hspace*{0pt}\\

Assume that there exists an open interval $I$ along which the switching function vanishes. Then for all $t\in I$,
\begin{equation*}
\langle p(t), f_1(v(t),w(t))\rangle = 0.
\end{equation*}
 By continuity, this means that either $v$ is constant and equals $V_{ChR2}$ on $I$ or $p_v$ vanishes on $I$. The constant case is not possible since it implies from the dynamical system ($FHN$) that $w$ would also be constant on $I$, but $(V_{ChR2},w)$ is not an equilibrium point of the uncontrolled system, for any $w\in\R$. Then, necessarily, $p_v$ vanishes on $I$. This implies that $\dot{p}_v$ also vanishes and from ($FHN_{adj}$), $p_w$ vanishes on $I$. This is incompatible from the Pontryagin maximum principle. 

We showed that the reduced system does not present any singular extremals and from Theorem \ref{ReductionTheorem}, the original system $(FHN)$ does not either. The optimal control is then bang-bang and is given by the sign of the switching function of the original system. Taking into account that for all $t\in[0,t_f]$, $1-o(t)-d(t) > 0$ we get 

\begin{equation*}
u^*(t) = u_{max} \mathbf{1}_{p_o(t)>0}, \quad \forall t\in[0,t_f].
\end{equation*}

Finally, to show that the first arc correspond to a maximal control, suppose that $u^*(0) = 0$. Then system ($FHN$) stays in its resting state, contradicting time optimality.

\end{proof}

We implement the direct method for this problem with a targeted action potential $V_s := 1.5mV$ and a control evolving in $[0,0.1]$. The numerical values of the constants $(a,b,c)$ are set to the usual values $(0.7,0.8,0.08)$. Since this model is not physiological, we chose the values for the constants $C$, $g_{ChR2}$, $V_ChR2$ and $u_{max}$ quite arbitrarily, with the constraint that the behavior of the control system should not stray away from the uncontrolled system. When the control is off, the system stays at rest, as seen on Figure \ref{NoStimulation2}.

\begin{figure}[!ht]
\begin{center}
\includegraphics*[scale=0.55]{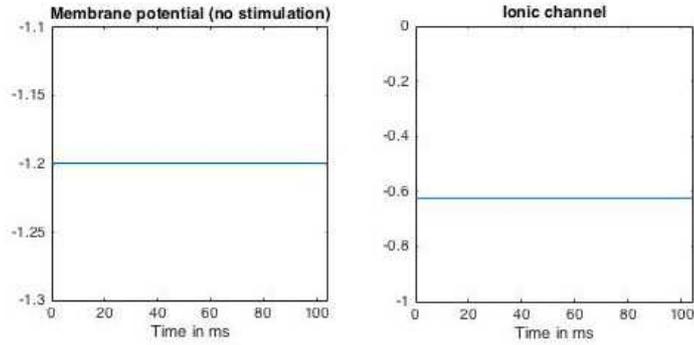}
\end{center}
\caption{In the absence of stimulation, the neuron stays in its resting state.}
\label{NoStimulation2}
\end{figure}

We represent on Figure \ref{FHN_ChR2_3States} the evolution of the optimal trajectory of the membrane potential and the optimal control. As predicted, the optimal control is bang-bang and starts with a maximal arc. It has a unique switching time which means that there is no need to keep the light on all the way to the spike, an interesting fact for the controller. This optimal control can be qualified as physiological, the light must stay on until a point where the system is "launched" toward the spike and no further illumination is required.

%\begin{figure}[!ht]
%\begin{center}
%\includegraphics[scale=0.6]{FHN_ChR2_3States2}
%\end{center}
%\caption{Optimal trajectory and control for the FHN-ChR2-3-states model.}
%\label{FHN_ChR2_3States}
%\end{figure}

\begin{figure}[!ht]
\begin{center}
\includegraphics[scale=0.5]{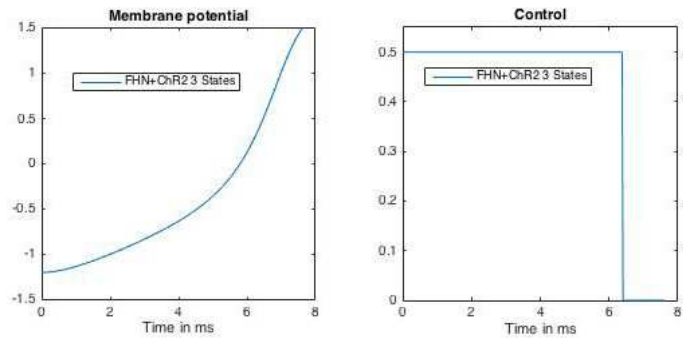}
\end{center}
\caption{Optimal trajectory and control for the FHN-ChR2-3-states model with $u_{max}=0.5\,\mathrm{ms}^{-1}$.}
\label{FHN_ChR2_3States}
\end{figure}

\paragraph{The ChR2-4-states model}\hspace*{0pt}\\

The ChR2-4-states model gives the same shape of optimal trajectory and control. We can compare the two ChR2 models and observe the results for different values of $u_{max}$ on Figure \ref{FHN_ChR2_3and4States}. The ChR2-4-states model outperforms the ChR2-3-states on two scales. It leads to a faster spike while requiring less time in the light to fire. This phenomenon seems to be independent of the maximal value of the control. The gain is of around $6\%$ in the four cases.

\clearpage

\begin{figure}[!ht]
\centering
\begin{tabular}{c}
   a) \includegraphics[scale=0.52]{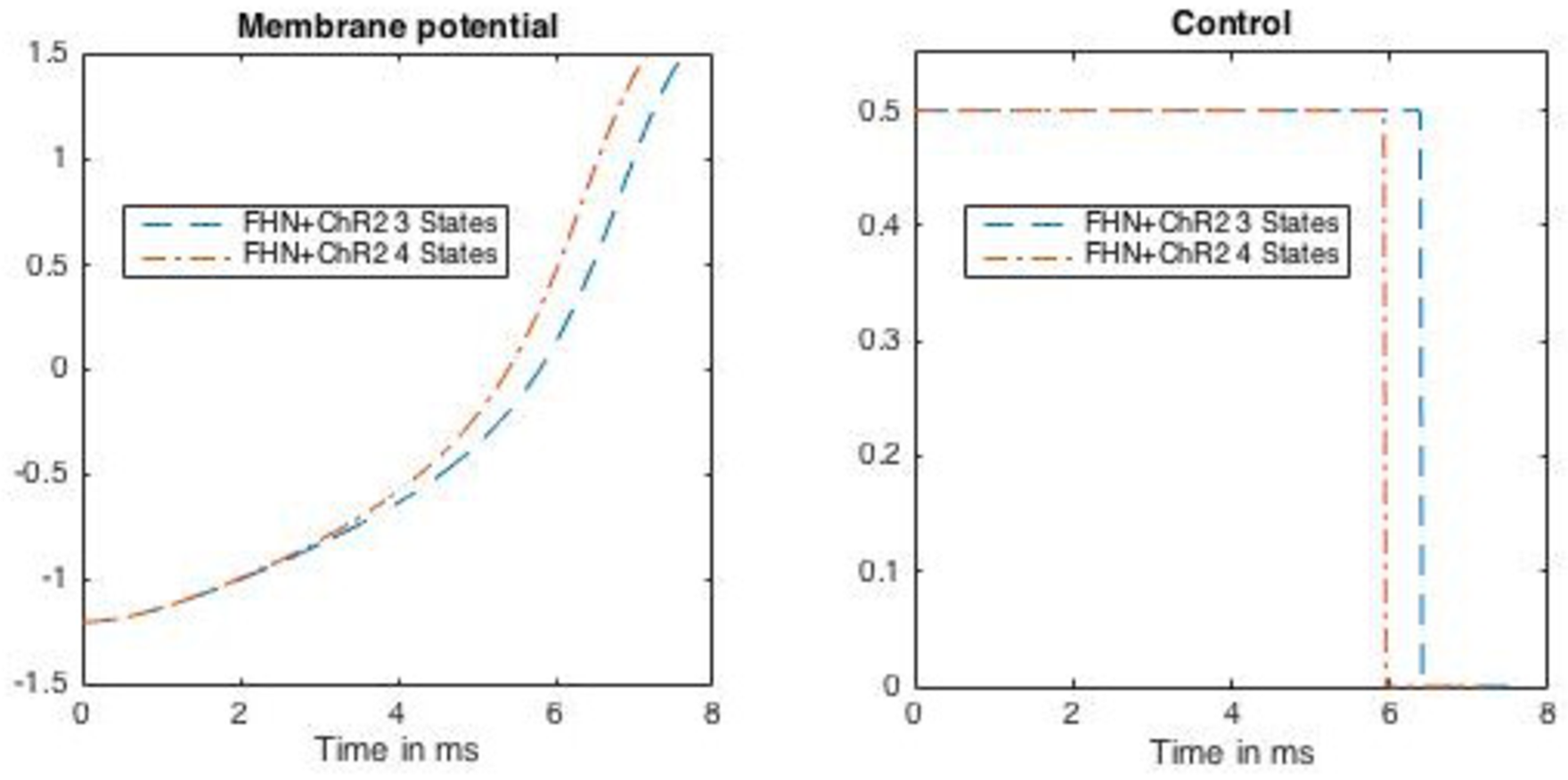} \\
   b) \includegraphics[scale=0.52]{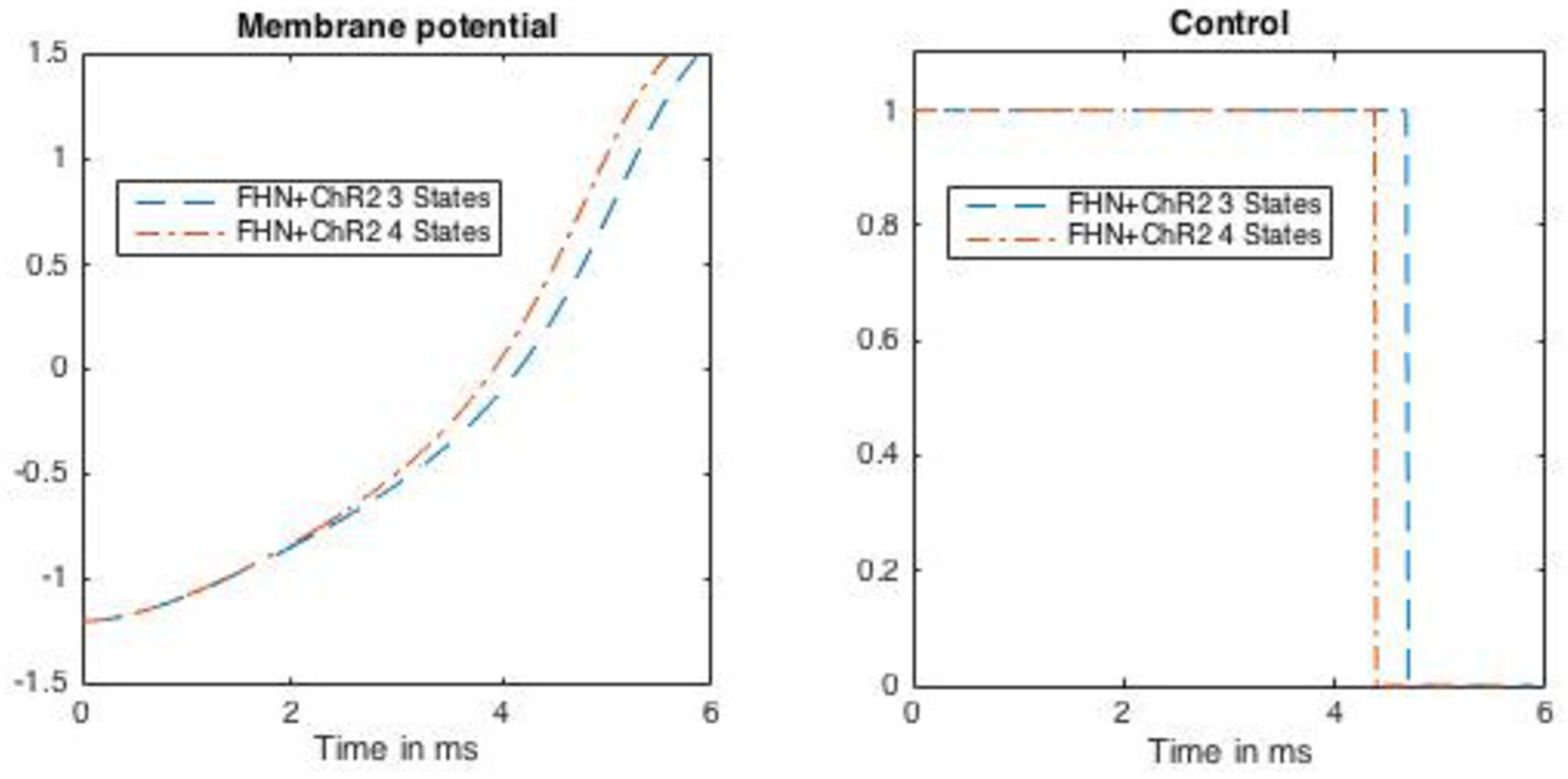} \\
   c) \includegraphics[scale=0.52]{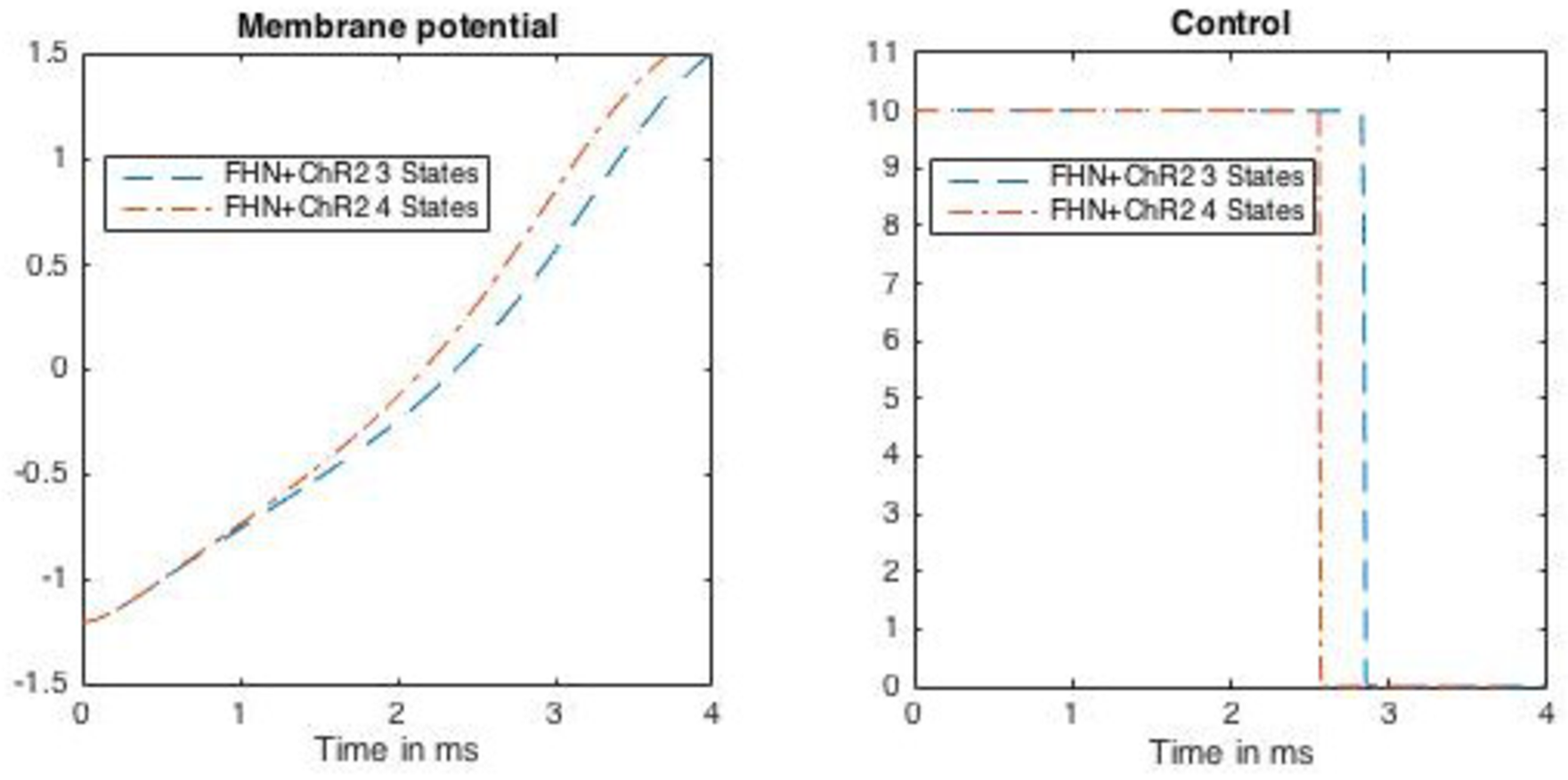} \\
   d) \includegraphics[scale=0.52]{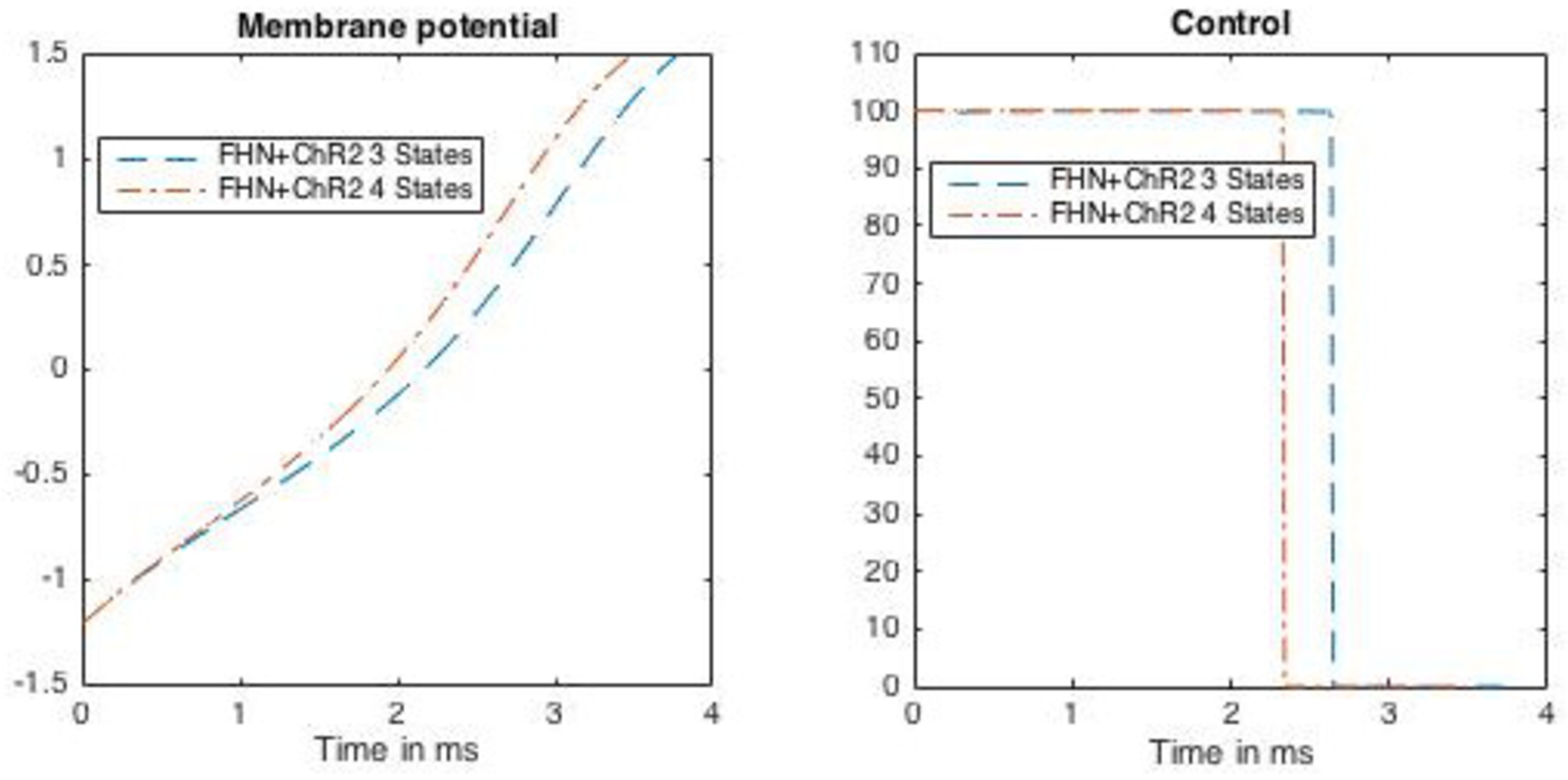}
\end{tabular}
\caption{Optimal trajectory and bang-bang optimal control for the FHN-ChR2-3-states and FHN-ChR2-4-states models with $u_{max}=$ a) $0.5$, b) $1$, c) $10$, d ) $100\,\mathrm{ms}^{-1}$.}
\label{FHN_ChR2_3and4States}
\end{figure}

\pagebreak
%\begin{figure}[!ht]
%\begin{center}
%\includegraphics[scale=0.6]{FHN_ChR2_4States_Sing2}
%\end{center}
%\caption{Optimal trajectory and singular optimal control for the FHN-ChR2-4-states model.}
%\label{FHN_ChR2_4States_Sing}
%\end{figure}

\subsection{The Morris-Lecar model}

The Morris-Lecar model is a reduced conductance-based model taking into account a $Ca^{2+}$ current for excitation and a $K^+$ current for recovery (\cite{MorrisLecar}). It comes from the experimental study of the oscillatory behavior of the membrane potential in the barnacle muscle. The original model is of dimension 3, but it is conveniently and commonly reduced to a two-dimensional model by invoking the fast dynamics of the $Ca^{2+}$ conductance in front of the other variables. This conductance is then replaced by its steady-state. 

\paragraph{The ChR2-3-states model}\hspace*{0pt}\\

\noindent The ChR2-3-states controlled Morris-Lecar model is given by

\begin{equation*}
(ML)\left\{
\begin{aligned}
\dot{\nu}(t) & = \frac{1}{C}\Big(g_K\omega(t)(V_K-\nu(t)) + g_{Ca}m_{\infty}(\nu(t))(V_{Ca}-\nu(t)) \\
&\qquad \qquad+g_{ChR2}o(t)(V_{ChR2}-\nu(t)) + g_L(V_L-\nu(t))\Big),\\ 
\dot{\omega}(t) & =  \alpha(\nu(t))(1-\omega(t)) - \beta(\nu(t))\omega(t),\\
\dot{o}(t) & = u(t)(1-o(t)-d(t)) - K_do(t),\\
\dot{d}(t) & = K_do(t) - K_rd(t),
\end{aligned}
\right.
\end{equation*}
with

\begin{align*}
m_{\infty}(\nu) &= \frac{1}{2}\left(1+\tanh\left(\frac{\nu-V_1}{V_2}\right)\right),\\
\alpha(\nu) &=\frac{1}{2}\phi\cosh\left(\frac{\nu-V_3}{2V_4}\right)\left(1+\tanh\left(\frac{\nu-V_3}{V_4}\right)\right),\\
\beta(\nu) &=\frac{1}{2}\phi\cosh\left(\frac{\nu-V_3}{2V_4}\right)\left(1-\tanh\left(\frac{\nu-V_3}{V_4}\right)\right),
\end{align*}
where $\nu$ is the membrane potential, $\omega$ is the probability of opening of a $K^+$ channel and $m_{\infty}(\nu)$ represent the steady state of the probability of opening of a $Ca^{2+}$ channel. The numerical constants of the model are given in Appendix \ref{AppendixML}. The adjoint equations read 

\begin{equation*}
(ML_{adj})\left\{
\begin{aligned}
\dot{p}_{\nu}(t) & =  \frac{1}{C}p_{\nu}(t)\Big(g_K\omega(t)+g_{Ca}m_{\infty}(\nu(t))+g_{ChR2}o(t)+g_L-g_{Ca}m'_{\infty}(\nu(t))\Big)\\
&\qquad -p_{\omega}(t)\Big(\alpha'(\nu(t))(1-\omega(t))-\beta'(\nu(t))\omega(t)\Big),\\ 
\dot{p}_{\omega}(t) & = -\frac{1}{C}g_K(V_K-\nu(t))p_{\nu}(t) + \Big(\alpha(\nu(t))+\beta(\nu(t))\Big)p_{\omega}(t),\\
\dot{p}_o(t) & = -\frac{1}{C}g_{ChR2}(V_{ChR2}-\nu(t))p_{\nu}(t) + (u(t)+K_d)p_o(t) -K_dp_d(t),\\
\dot{p}_d(t) &= u(t)p_o(t) + K_rp_d(t),
\end{aligned}
\right.
\end{equation*}
and the switching function is again $\varphi(t) = (1-o(t)-d(t))p_o(t)$. Proposition \ref{PropML} gives the same conclusion as Proposition \ref{PropFHN} for the ChR2-controlled Morris-Lecar model.

\begin{proposition}\label{PropML}
The optimal control $u^*:\R_+\rightarrow U$ for the minimal time control of the Morris-Lecar model is bang-bang and given by 

\begin{equation*}
u^*(t) = u_{max} \mathbf{1}_{p_o(t)>0}, \quad \forall t\in[0,t_f].
\end{equation*}
Furthermore, the optimal control begins with a bang arc of maximal value

\[
\exists t_1\in [0,t_f], u^*(t)=u_{max}, \forall t\in[0,t_1]. 
\]

\end{proposition}

\begin{proof}
We apply the result of Theorem \ref{ReductionTheorem} and study the existence of singular extremals for the following reduced system

\begin{equation*}
(ML')\left\{
\begin{aligned}
\dot{\nu}(t) & = \frac{1}{C}\Big(g_K\omega(t)(V_K-\nu(t)) + g_{Ca}m_{\infty}(\nu(t))(V_{Ca}-\nu(t)) \\
&\qquad \qquad +g_{ChR2}u(t)(V_{ChR2}-\nu(t)) + g_L(V_L-\nu(t))\Big),\\ 
\dot{\omega}(t) & =  \alpha(\nu(t))(1-\omega(t)) - \beta(\nu(t))\omega(t),
\end{aligned}
\right.
\end{equation*}
The adjoint equations for this system are 

\begin{equation*}
(ML'_{adj})\left\{
\begin{aligned}
\dot{p}_{\nu}(t) & =  \frac{1}{C}p_{\nu}(t)\Big(g_K\omega(t)+g_{Ca}m_{\infty}(\nu(t))+g_{ChR2}u(t)+g_L-g_{Ca}m'_{\infty}(\nu(t))\Big)\\
&\qquad -p_{\omega}(t)\Big(\alpha'(\nu(t))(1-\omega(t))-\beta'(\nu(t))\omega(t)\Big),\\ 
\dot{p}_{\omega}(t) & = -\frac{1}{C}g_K(V_K-\nu(t))p_{\nu}(t) + (\alpha(\nu(t))+\beta(\nu(t)))p_{\omega}(t),\\
\end{aligned}
\right.
\end{equation*}
The vector fields defining the affine system ($ML'$) are

\begin{equation*}
\begin{aligned}
f_0(\nu,\omega) & =  \frac{1}{C}\Big(g_K\omega(V_K-\nu) + g_{Ca}m_{\infty}(\nu)(V_{Ca}-\nu) + g_L(V_L-\nu)\Big)\partial_{\nu}\\
&\qquad \qquad +  \Big(\alpha(\nu)(1-\omega) - \beta(\nu)\omega\Big)\partial_{\omega}\\ 
f_1(\nu,\omega) & = \frac{1}{C}g_{ChR2}(V_{ChR2}-v)\partial_{\nu}\\
\end{aligned}
\end{equation*}
For the reduced system, the switching function is given by 

\begin{equation*}
\phi(t) = 	\langle p(t), f_1(\nu(t),\omega(t))\rangle = \frac{1}{C}g_{ChR2}(V_{ChR2}-\nu(t))p_{\nu}(t).
\end{equation*}

\paragraph{Investigation of singular trajectories}\hspace*{0pt}\\

Assume that there exists an open interval $I$ along which the switching function vanishes. Then for all $t\in I$,
\begin{equation*}
\langle p(t), f_1(v(t),w(t))\rangle = 0.
\end{equation*}
As for the FitzHugh-Nagumo model, there is no $\omega\in[0,1]$ such that $(V_{ChR2},\omega)$ is an equilibrium point of the uncontrolled Morris-Lecar model, so that necessarily $p_{\omega}$ vanishes on $I$. From ($ML'$) we deduce that for all $t\in I$,

\[
p_{\omega}(t)\Big(\alpha'(\nu(t))(1-\omega(t))-\beta'(\nu(t))\omega(t)\Big) = 0,
\]
and since $p$ cannot vanish on $I$ then  

\[
\alpha'(\nu(t))(1-\omega(t))-\beta'(\nu(t))\omega(t) = 0.
\]
This means that the singular extremal is localized in the domain $A$ of $\R^2$ given by 

\begin{equation*}
A:= \{(\nu,\omega)\in\R^2 | \alpha'(\nu)(1-\omega)-\beta'(\nu)\omega = 0 \}.
\end{equation*}
We can rewrite it in a more convenient way

\begin{equation*}
A= \left\{(\nu,\omega)\in\R^2 | \omega = \frac{\alpha'(\nu)}{\alpha'(\nu)+\beta'(\nu)} \text{ and } \nu\neq V_3 \right\}.
\end{equation*}
Domain $A$ is represented on Figure \ref{SingularExtremalFigure} below and it is easy to see that any trajectory of the dynamical system ($ML'$) has an empty intersection with $A$ because for all $(\nu,\omega)\in A$, $\omega \in ]-\infty,0[\cup]1,+\infty[$, whereas the second component of the trajectory always stays in $[0,1]$.

\begin{figure}[!ht]
\begin{center}
\includegraphics[scale=0.5]{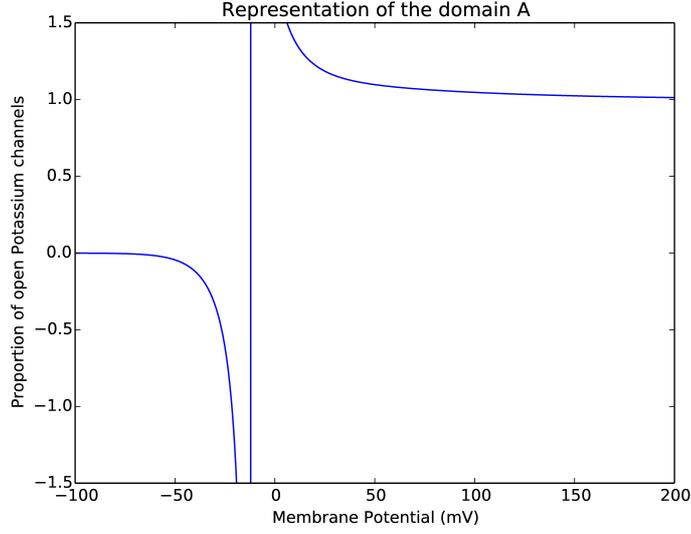}
\end{center}
\caption{Representation of the manifold in which a singular trajectory must evolve.}
\label{SingularExtremalFigure}
\end{figure}
The end of the proof is similar to the proof of Proposition \ref{PropFHN}.

\end{proof}

\begin{remark}

Let us briefly show how the investigation of singular trajectories for the complete system before reduction is much more difficult. To do so, consider the controlled Morris-Lecar model ($ML$) with its system of adjoint equations ($ML_{adj}$) and the vector fields defined for $x=(\nu,\omega,o,d)\in\R^4$ by

\begin{align*}
F_0(x) &:= \frac{1}{C}\Big(g_K\omega(V_K-\nu) + g_{Ca}m_{\infty}(\nu)(V_{Ca}-\nu) + og_{ChR2}(V_{ChR2}-\nu) + g_L(V_L-\nu)\Big)\partial_{\nu} \\
&\qquad \qquad \Big(\alpha(\nu)(1-\omega) - \beta(\nu)\omega\Big)\partial_{\omega} -K_do\partial_o + (K_do-K_rd)\partial_d,
\end{align*}

and 

\begin{equation*}
F_1(x)  =(1-o-d)\partial_{o}.
\end{equation*}

\begin{proposition}\label{PropositionReductionBenefit}
Let ($x,p,u$) be a singular extremal of $(ML)-(ML_{adj})$ on an open interval $I$ of $[0,t_f]$. Then, without any further assumption, 
\begin{equation*}
\langle p(t),\mathrm{ad}^k_{F_0}F_1(x(t))\rangle \equiv 0,\quad \langle p(t),\mathrm{ad}^k_{F_1}F_0(x(t)) \rangle \equiv 0, \quad \langle p(t),[F_1,\mathrm{ad}^2_{F_0}F_1](x(t)) \rangle \equiv 0,
\end{equation*}
on $I$ for all $k\in\{1,2,3\}$.
\end{proposition}

Keeping in mind that we already proved that there is no optimal singular control, if we consider the system before reduction, Proposition \ref{PropositionReductionBenefit} means that we need to consider the following system of equations to rule out optimal singular extremals

\begin{align*}
\langle p,[F_0,\mathrm{ad}^3_{F_1}F_0]\rangle + u\langle p,\mathrm{ad}^4_{F_1}F_0 \rangle & \equiv 0,\\
\langle p,[F_0,[F_1,\mathrm{ad}^2_{F_0}F_1]]\rangle + u\langle p,\mathrm{ad}^2_{F_1}(\mathrm{ad}^2_{F_0}F_1) \rangle & \equiv 0,\\
\langle p,\mathrm{ad}^4_{F_0}F_1\rangle + u\langle p,[F_1,\mathrm{ad}^3_{F_0}F_1] \rangle & \equiv 0,
\end{align*}
on $I$.

\begin{proof}[Proof of Proposition \ref{PropositionReductionBenefit}]

Let $t\in I$. From the equalities $\langle p(t),F_1(x(t)) \rangle = 0$ and \\ $\langle p(t),[F_0,F_1](x(t)) \rangle = 0$ we infer that 

\begin{equation}\label{singularML}
\left\{
\begin{aligned}
&p_o(t) = 0,\\
&\frac{1}{C}g_{ChR2}(V_{ChR2}-\nu(t))p_v(t) + K_dp_d(t) = 0.
\end{aligned}
\right.
\end{equation}

It can also be proved that $\mathrm{ad}^3_{F_1}F_0 = -[F_0,F_1]$. The rest of the equalities are all given by (\ref{singularML}).

\end{proof}

\end{remark}

For this model, we implemented the direct method with the numerical values of Appendices \ref{AppendixML} and \ref{AppendixChR2_3States}. The targeted action potential has been fixed to $30$mV.

The optimal control for the ChR2-3-states model is bang-bang and begins with a maximal arc. For the numerical values of Appendices \ref{AppendixML} and \ref{AppendixChR2_3States}, it displays three switching times. We represent on Figure \ref{MLCHR2_3States} the optimal trajectory of the membrane potential and the optimal control, for the physiological value of the maximal value control, computed in Appendix \ref{AppendixChR2_3States}, and also the trajectory obtained under constant maximal stimulation, just to observe that the optimal control obtained is indeed better than the constant maximal stimulation. If the difference is very small, of the order of a millisecond, the counter-intuitive stimulation still outperforms the constant maximal stimulation. In order to show that the difference between the counter-intuitive optimal stimulation and the constant maximal stimulation can be huge, we implement the direct method on a system with different numerical values for the constants of the Morris-Lecar model (the Type I neuron of \cite[Table 1]{Longtin}, see Table \ref{LongtinChR2}, in Appendix \ref{AppendixChR2_3States} ), and values for the ChR2-3-states model remaining unchanged, except for $V_{ChR2}=0.1$mV. The result is striking, the constant stimulation even fails to trigger a spike while the stimulation with three switching times makes the neuron fire (see Figure \ref{MLCHR2_3StatesNoSpike}). It is important to note that the presence of three switching times is not an intrinsic characteristic of the Morris-Lecar model itself. Indeed, we can find optimal controls with only two switches if we change the value for the equilibrium potential of the ChR2, keeping all the other constants of the model unchanged (Figure \ref{MLCHR2_3StatesTwoSwitches}). 
\bigskip
\begin{figure}[!ht]
\begin{center}
\includegraphics[scale=0.6]{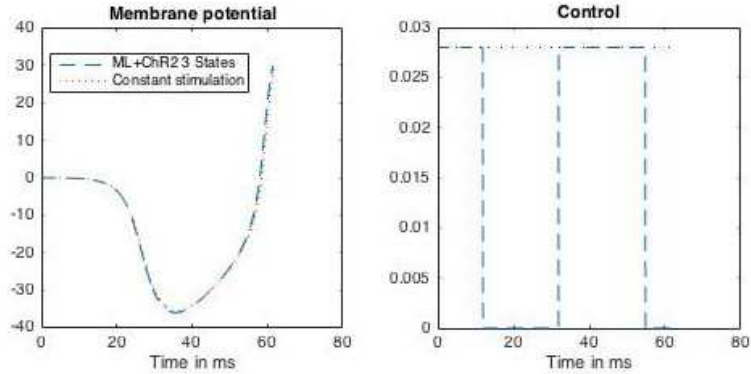}
\end{center}
\caption{Optimal trajectory and bang-bang optimal control for the ML-ChR2-3-states model.}
\label{MLCHR2_3States}
\end{figure}

\clearpage

\begin{figure}[!ht]
\begin{center}
\includegraphics[scale=0.55]{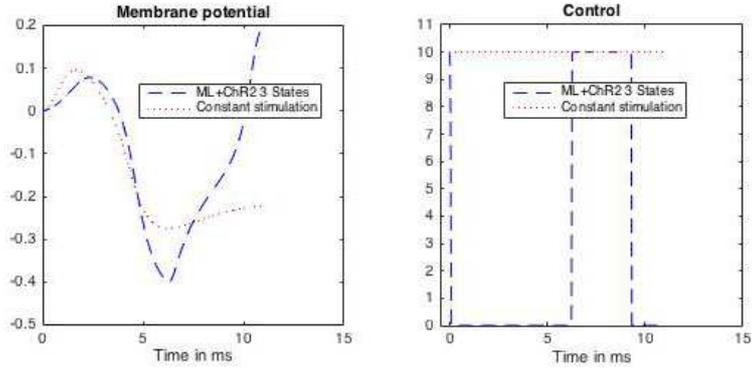}
\end{center}
\caption{Optimal trajectory and bang-bang optimal control for the ML-ChR2-3-states model with numerical values of \cite[Table 1]{Longtin}. The constant stimulation fails to trigger a spike.}
\label{MLCHR2_3StatesNoSpike}
\end{figure}

\begin{figure}[!ht]
\begin{center}
\includegraphics[scale=0.6]{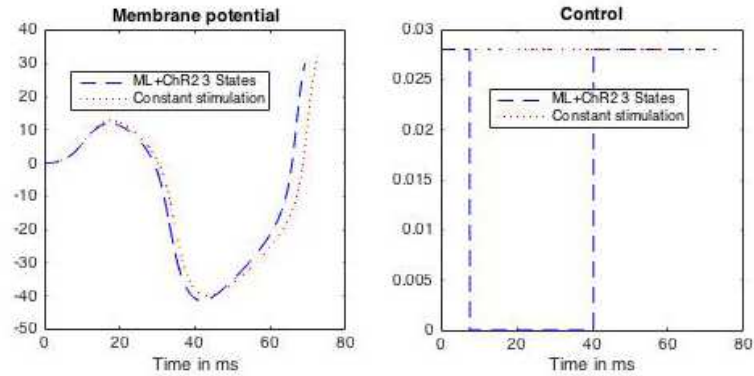}
\end{center}
\caption{Optimal trajectory and bang-bang optimal control for the ML-ChR2-3-states model with numerical values of Appendix \ref{AppendixML} and $VChR2=20$mV. The optimal control has only two switches.}
\label{MLCHR2_3StatesTwoSwitches}
\end{figure}

\paragraph{The ChR2-4-states model}\hspace*{0pt}\\

The shape of the optimal trajectory and control of the ChR2-4-states model correspond to the one of the ChR2-3-states model. Nevertheless, for small values of $u_{max}$, including the physiological value computed in Appendix \ref{AppendixChR2_3States}, the ChR2-3-states model outperforms the ChR2-4-states model whereas for larger values of $u_{max}$, the opposite happens (Figure \ref{ML_ChR2_3and4States}). The threshold where this phenomenon happens is around the value $u_{max} = 0.1$. Furthermore, the difference grows larger when $u_{max}$ increases. This is an usual behavior that suggests that the Morris-Lecar is less robust than the FitzHugh-Nagumo model, or the Hodgkin-Huxely models, as we are going to see.

\clearpage

\begin{figure}[!ht]
\centering
\begin{tabular}{c}
   a) \includegraphics[scale=0.52]{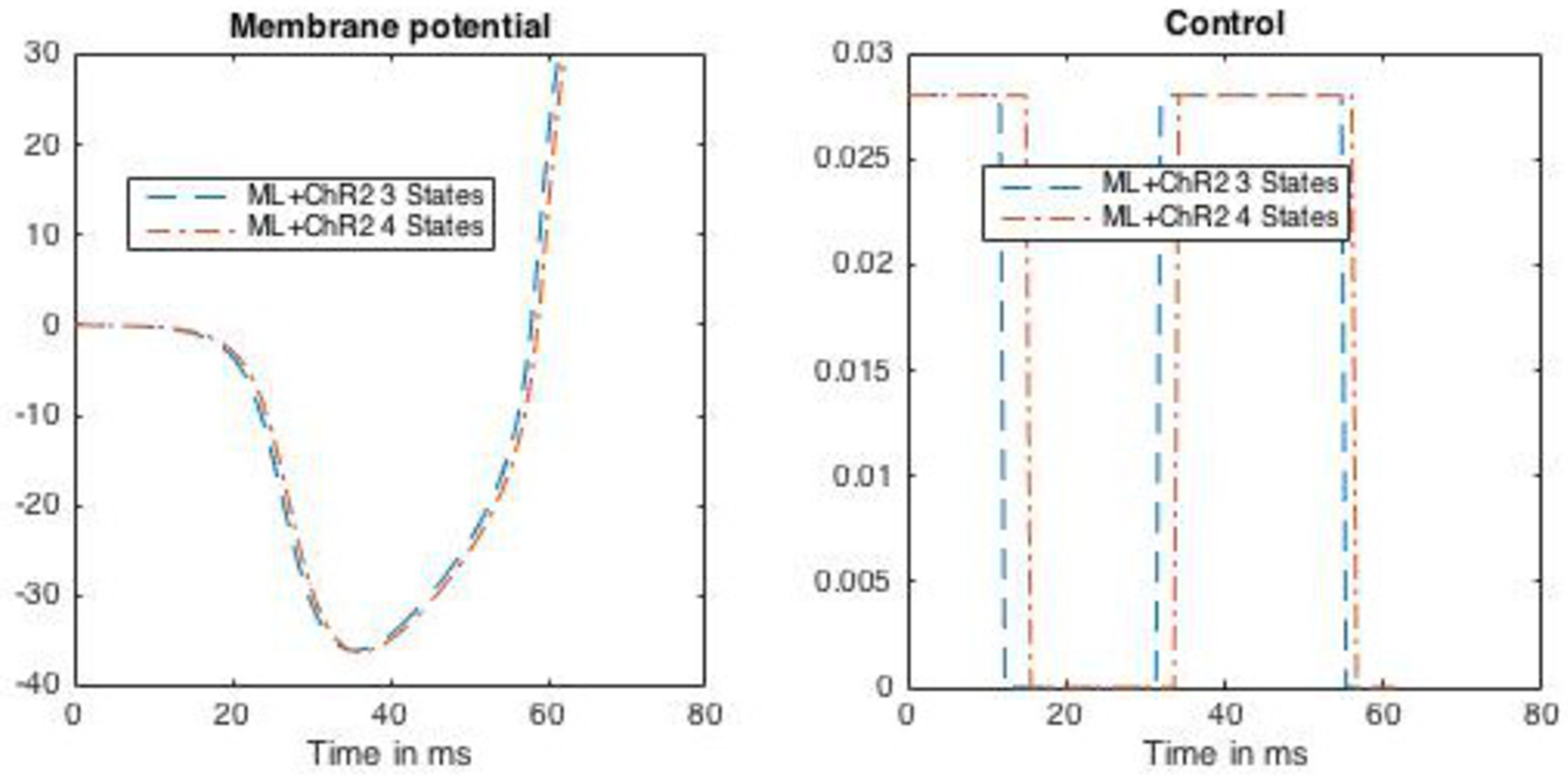} \\
   b) \includegraphics[scale=0.52]{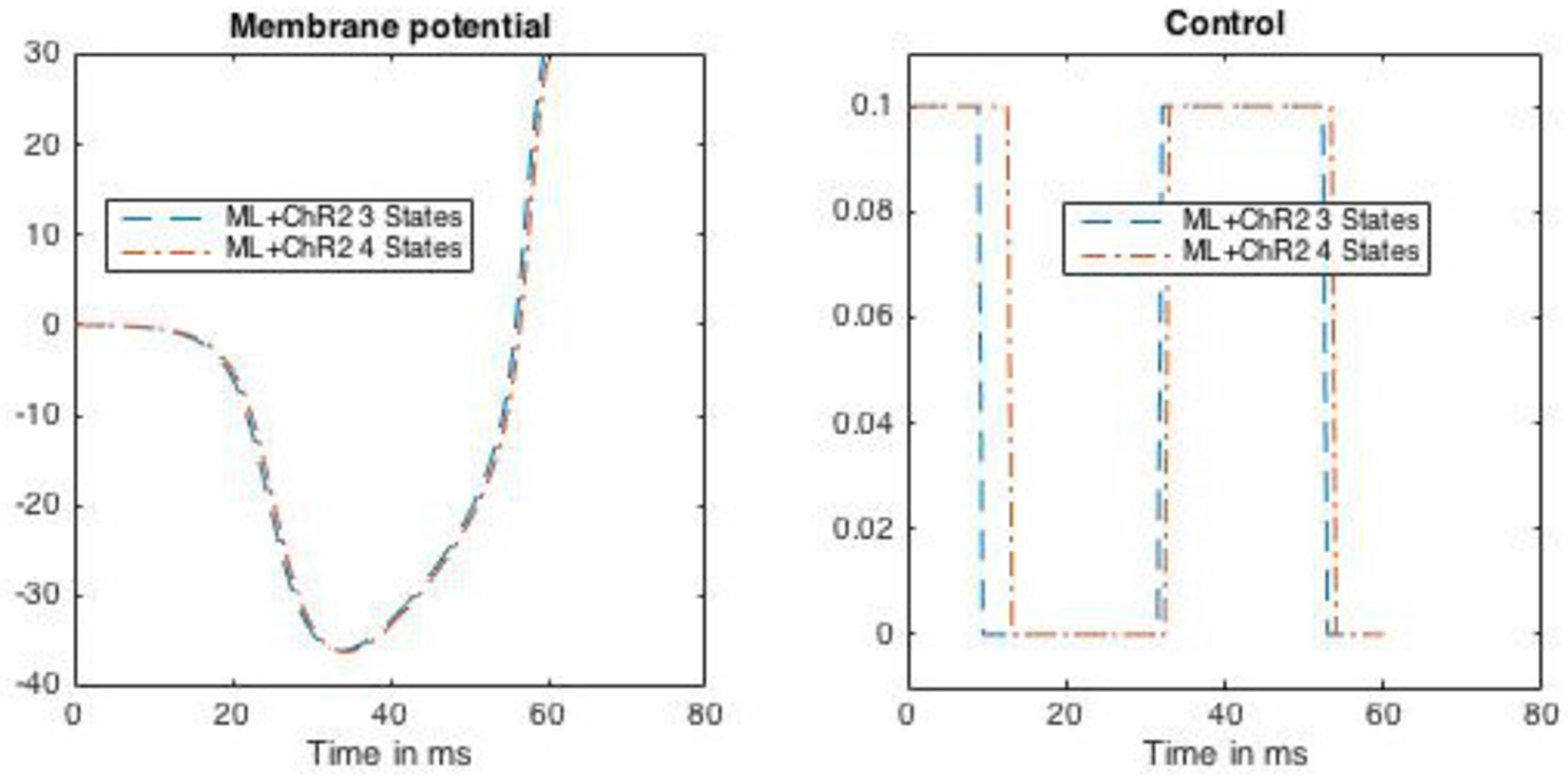} \\
   c) \includegraphics[scale=0.52]{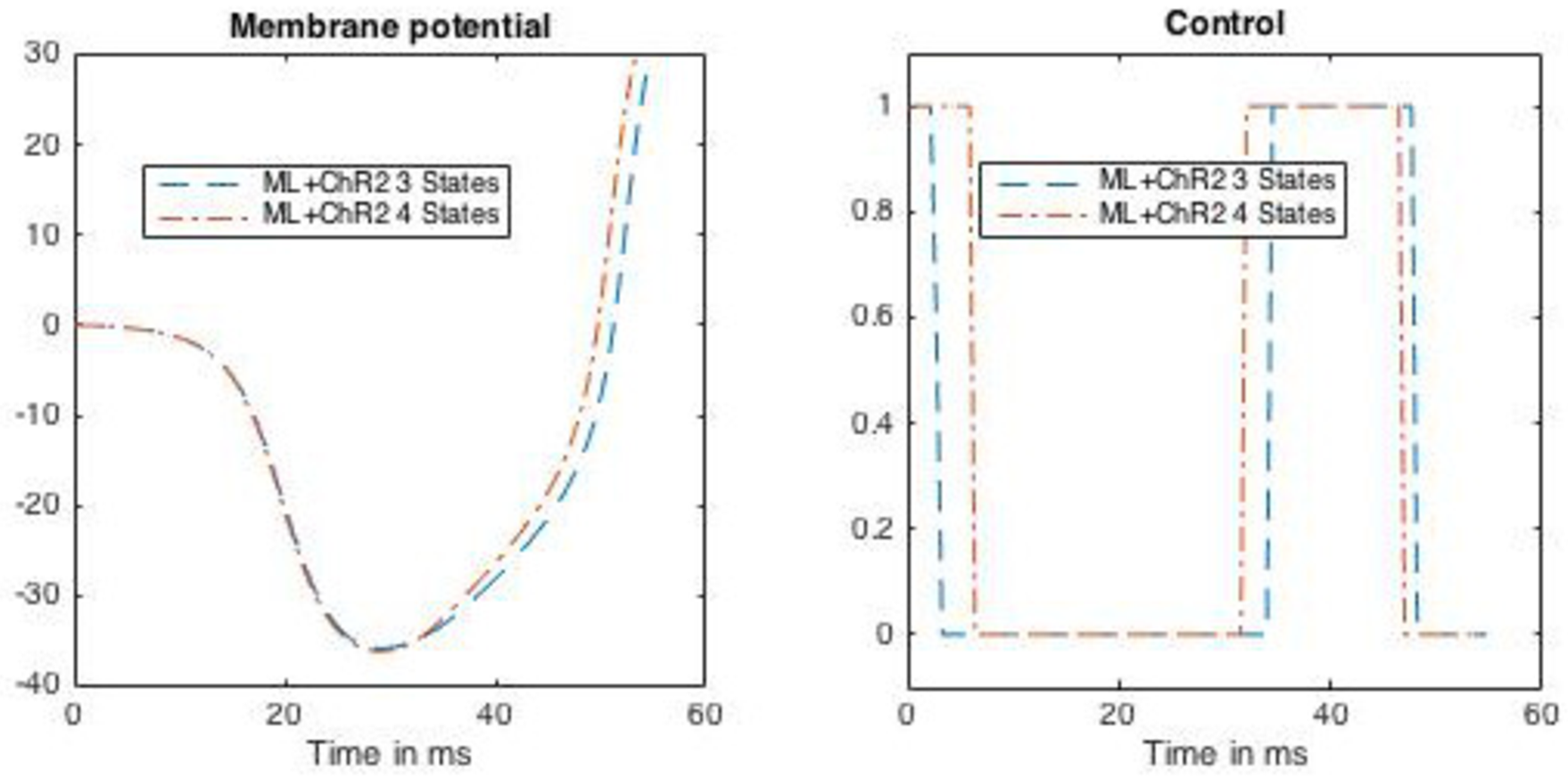} \\
   d)  \includegraphics[scale=0.52]{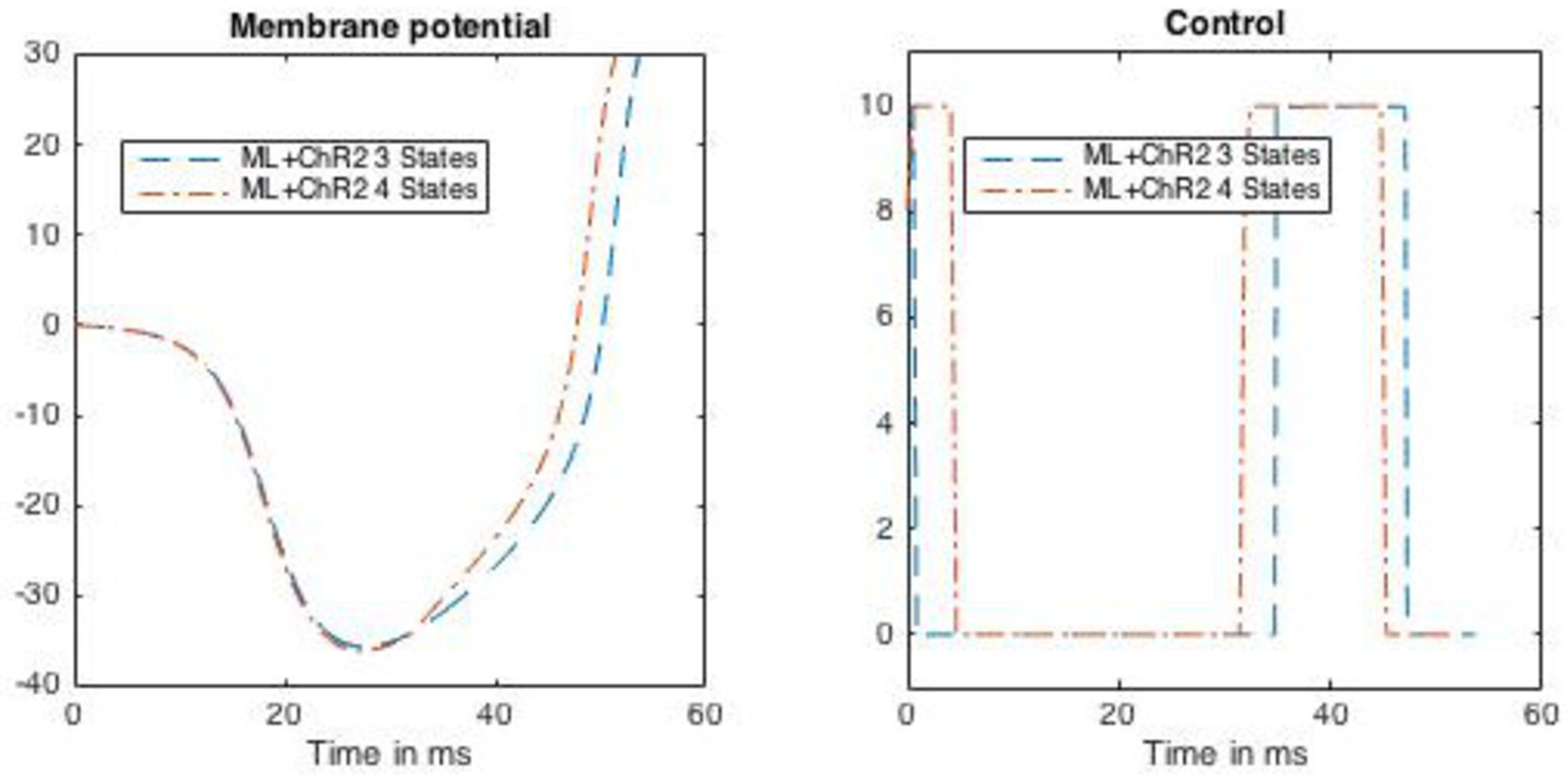}
\end{tabular}

\caption{Optimal trajectory and bang-bang optimal control for the ML-ChR2-3-states and ML-ChR2-4-states models with $u_{max}=$ a) $0.028$, b) $0.1$, c) $1$, d) $10\,\mathrm{ms}^{-1}$.}
\label{ML_ChR2_3and4States}
\end{figure}

%\begin{figure}[!ht]
%\begin{center}
%\includegraphics[scale=0.6]{ML_ChR2_4States2}
%\end{center}
%\caption{Optimal trajectory and bang-bang optimal control for the ML-ChR2-3-states model.}
%\label{MLCHR2_4States}
%\end{figure}
%
%\begin{figure}[!ht]
%\begin{center}
%\includegraphics[scale=0.55]{ML_ChR2_4StatesSing2}
%\end{center}
%\caption{Optimal trajectory and singular optimal control for the ML-ChR2-4-states model.}
%\label{MLCHR2_4States_Sing}
%\end{figure}

\subsection{The reduced Hodgkin-Huxley model}\label{subsecReducedHH}

Similarly to the reduction of the initial Morris-Lecar model, there exists a popular reduction of the Hodgkin-Huxley model to a 2-dimensional conductance-based model. This reduction is based on the observation that, on the one hand, the variable $m$ is much faster than the other two gating variables $n$ and $h$, and on the other hand, the variable $h$ is almost a linear function of the variable $n$ ($h\simeq a+bn$). These observations lead to a new system of equations derived from $(HH)$ by setting the variable $m$ in its stationary state $m(t)=m_{\infty}(t)$ and taking the variable $h$ as above.

\begin{equation*}
(HH_{2D})\left\{
\begin{aligned}
C\frac{\mathrm{d}V}{\mathrm{d}t} & = g_Kn^4(t)(V_K-V(t)) + g_{Na}m_{\infty}^3(V)(a+bn(t))(V_{Na}-V(t))\\
&\quad + g_L(V_L-V(t)),\\ 
\frac{\mathrm{d}n}{\mathrm{d}t} & =  \alpha_n(V(t))(1-n(t)) - \beta_n(V(t))n(t),
\end{aligned}
\right.
\end{equation*}

with $m_{\infty}(v) = \frac{\alpha_m(v)}{\alpha_m(v)+\beta_m(v)}$. It is important to note that, although the time constants of the ion channels have been mathematically investigated (see for example \cite{RubinWechselberger}), the approximation of the variable $h$ is purely based on observation, and not on a rigorous mathematical reduction. Nevertheless, if the linear approximation seems questionable when the membrane potential is held fixed (Figure \ref{hLinearFixVFigure}), it becomes quite remarkable when the whole system (HH) is considered as in Figure \ref{hLinearPeriodicFigure} for a periodic behavior and Figure \ref{hLinearTransitoryFigure} for a transitory behavior, with different initial membrane potentials $V_0$. The different behaviors are obtained by tuning the external current $I_{ext}$ that is applied.

\begin{figure}[!ht]
\begin{center}
\includegraphics[width=10cm]{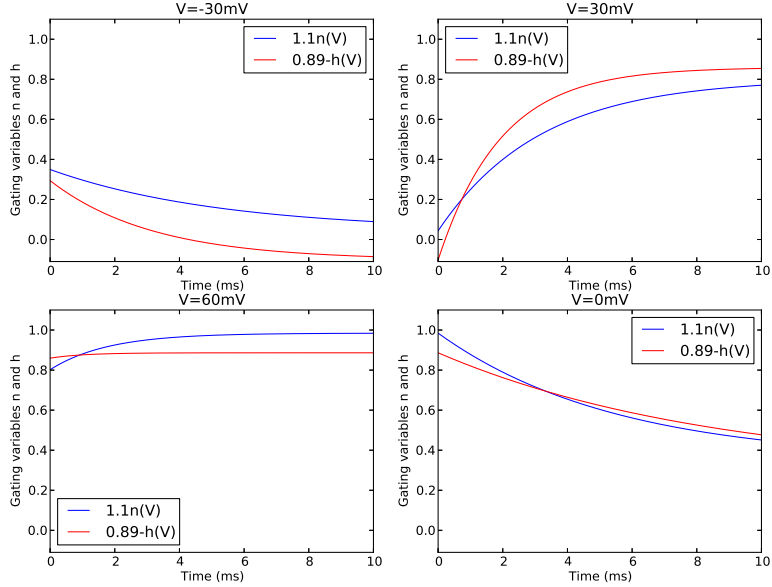}
\end{center}
\caption{Linear approximation of the variable $h$ when the membrane potential is held fixed at $-30$, $0$, $30$ and $60$mV.}
\label{hLinearFixVFigure}
\end{figure}

\begin{figure}[!ht]
\begin{center}
\includegraphics[width=10cm]{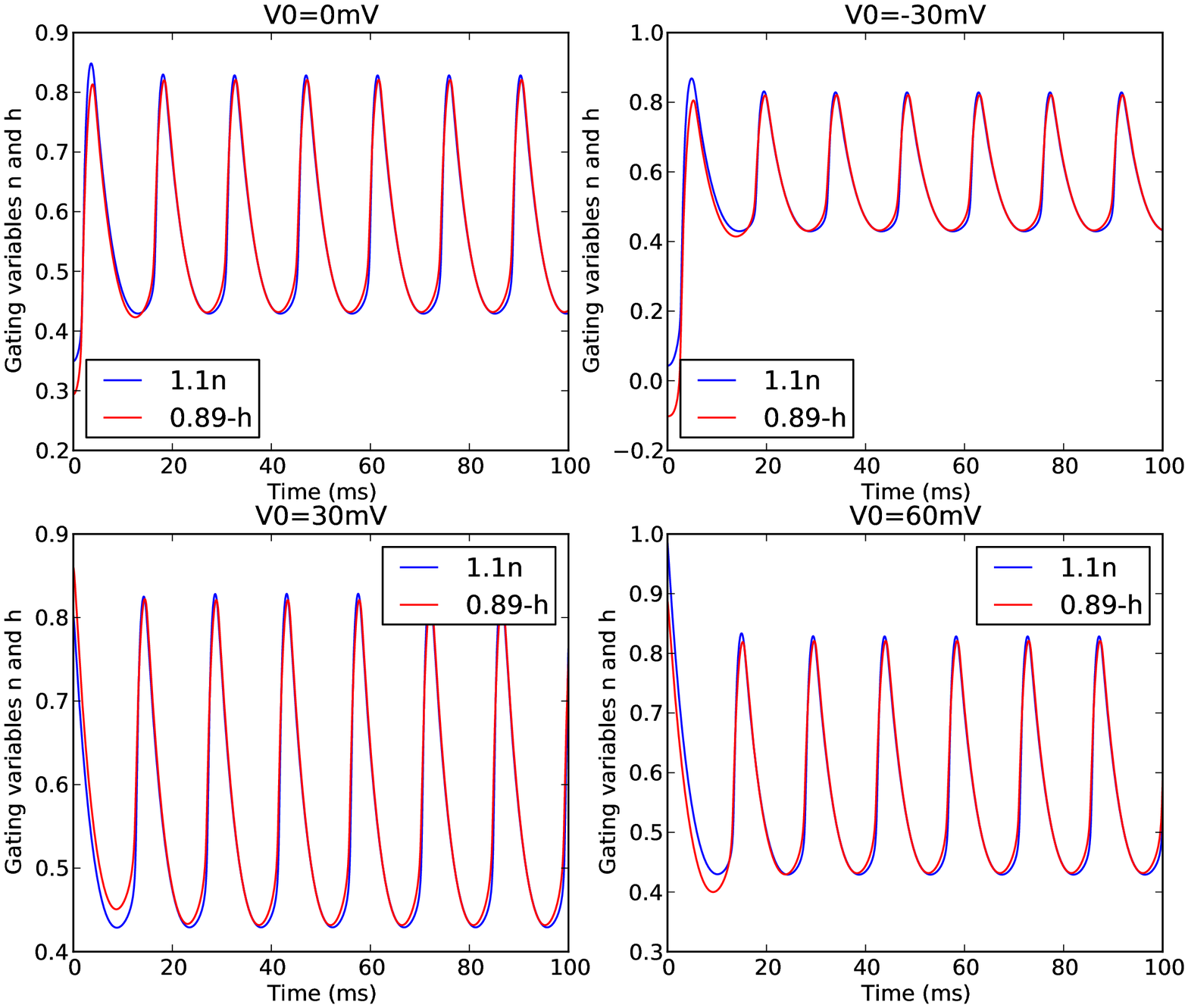}
\end{center}
\caption{Linear approximation of the variable $h$ for a periodic behavior of system ($HH$) and initial membrane potential of $-30$, $0$, $30$ and $60$mV.}
\label{hLinearPeriodicFigure}
\end{figure}

\begin{figure}[!ht]
\begin{center}
\includegraphics[width=10cm]{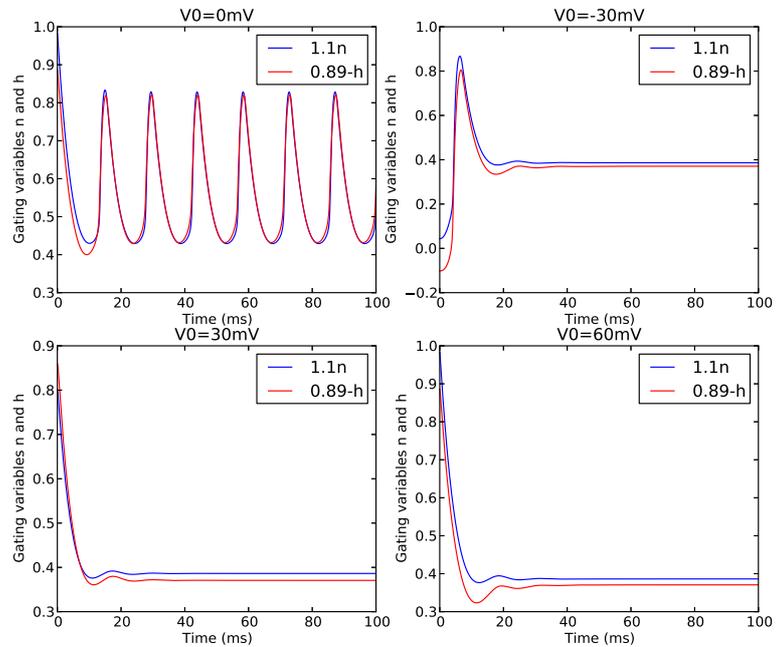}
\end{center}
\caption{Linear approximation of the variable $h$ for a transitory behavior of system ($HH$) and initial membrane potential of $-30$, $0$, $30$ and $60$mV.}
\label{hLinearTransitoryFigure}
\end{figure}

\clearpage

\paragraph{The ChR2-3-states model}\hspace*{0pt}\\

In terms of singular controls, this model behaves similarly to the Morris-Lecar model. There is no singular extremal for the same reasons, and the optimal control is bang-bang with the same expression (the proof is exactly the same). The direct method is implemented with the numerical values of Appendices \ref{AppendixHH} and \ref{AppendixChR2_3States}, the targeted action potential has been fixed to $90$mV. The optimal control is physiological here and has in fact no switching time, the light has to be on all the way to the spike (see Figure \ref{HH2DCHR2_3States}). 

\begin{figure}[!ht]
\begin{center}
\includegraphics[scale=0.6]{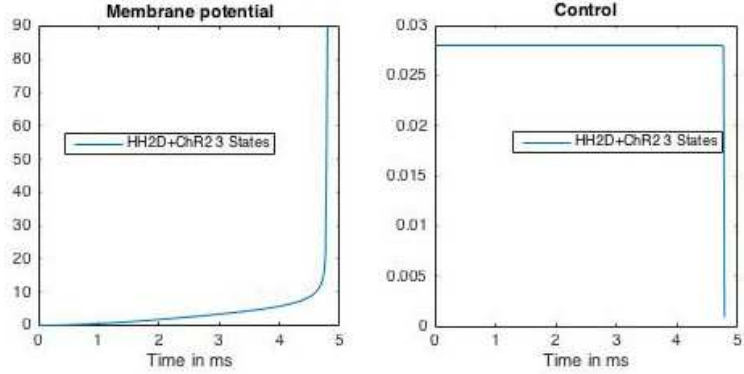}
\end{center}
\caption{Optimal trajectory and bang-bang optimal control for the HH2D-ChR2-3-states model.}
\label{HH2DCHR2_3States}
\end{figure}

\paragraph{The ChR2-4-states model}\hspace*{0pt}\\

The ChR2-4-states model is interesting because it shows that the Hodgkin-Huxley behaves in the opposite way of the Morris-Lecar model. Indeed, the ChR2-4-states model slightly outperforms the ChR2-3-states model, and requires less light, for small values of $u_{max}$, including the physiological value of $u_{max}=0.028$. Furthermore, when $u_{max}$ increases, the 3-states and 4-states models exactly match, both in terms of optimal trajectory and optimal control (Figure \ref{HH2D_ChR2_3and4States}). This means that the ChR2-3-states model is a good approximation of the ChR2-4-states model, in terms of optimal control, for the reduced Hodgkin-Huxley. This is a nice property since the ChR2-3-states is theoretically tractable in terms of singular controls. 

\clearpage

\begin{figure}[!ht]
\centering
\begin{tabular}{c}
   a) \includegraphics[scale=0.54]{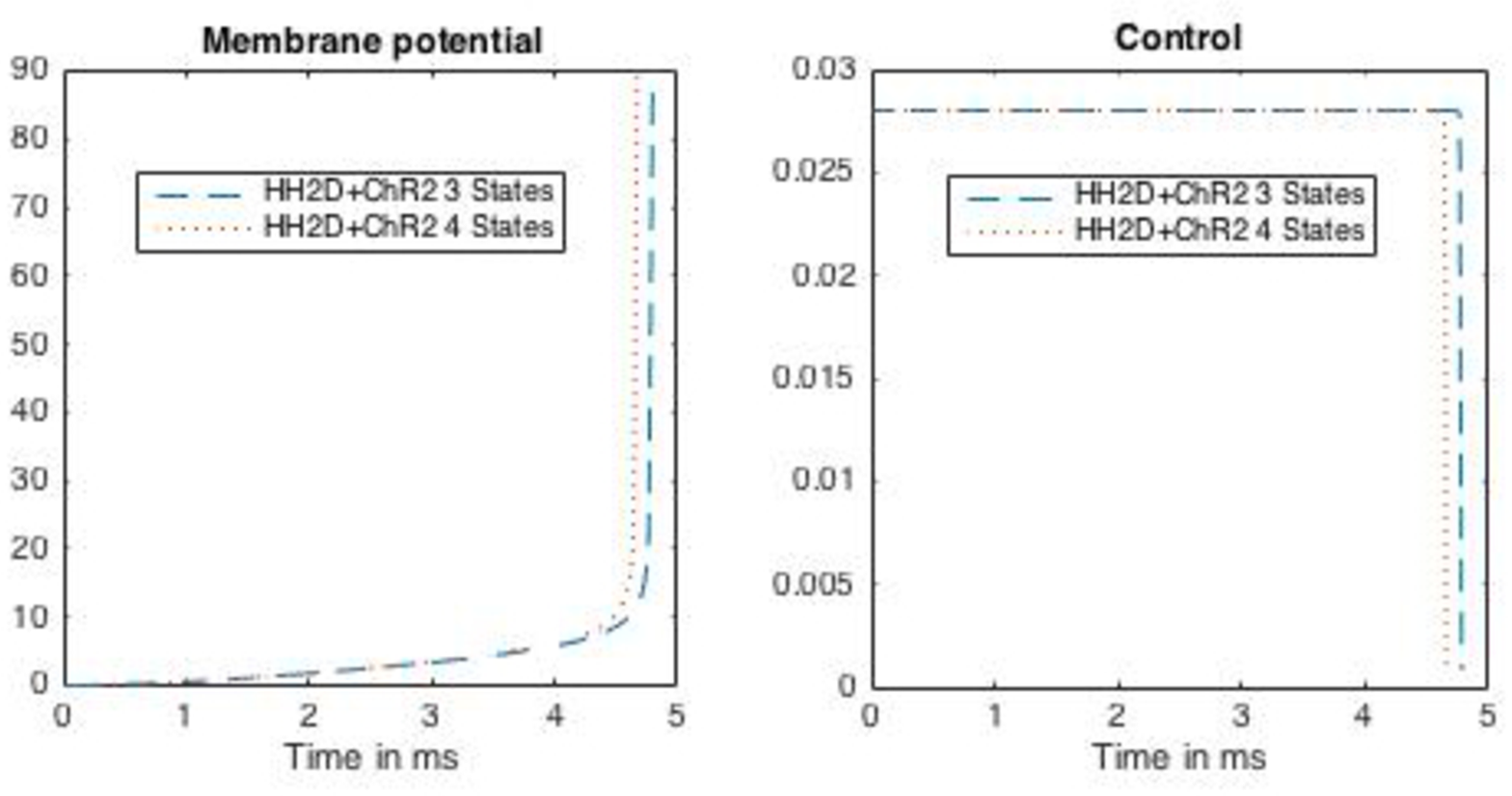} \\
   b) \includegraphics[scale=0.51]{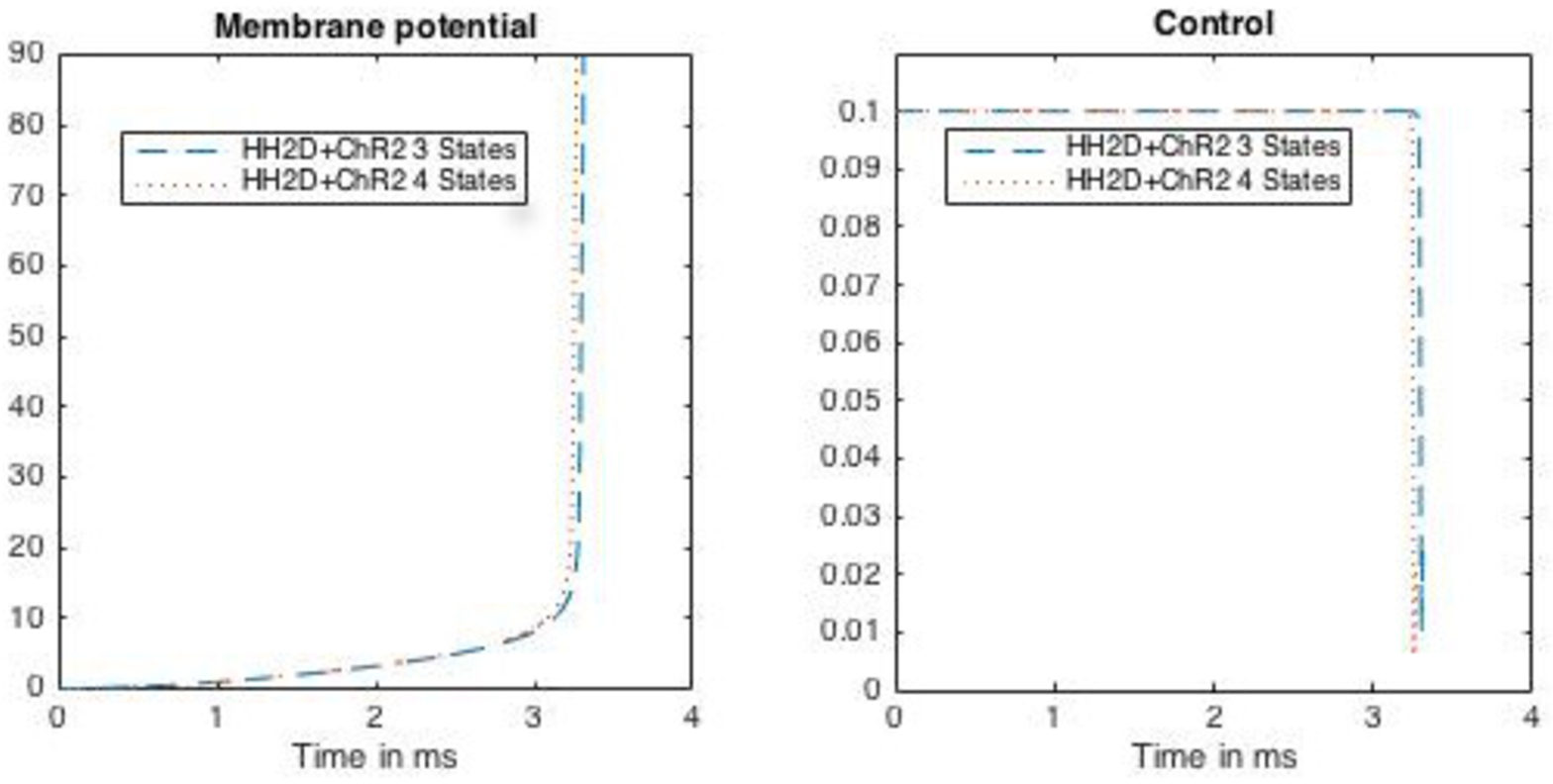} \\
   c) \includegraphics[scale=0.51]{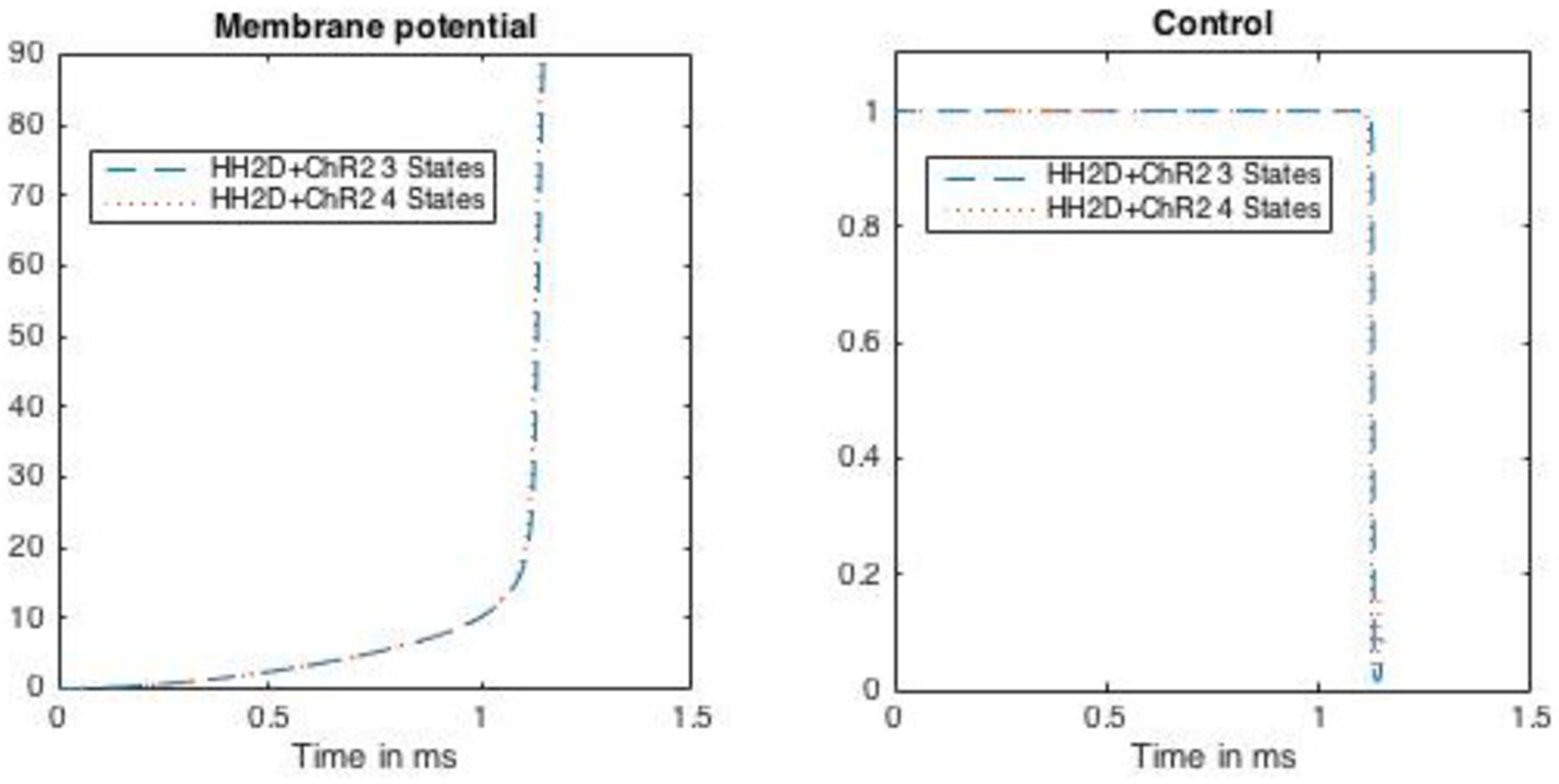} \\
   d) \includegraphics[scale=0.51]{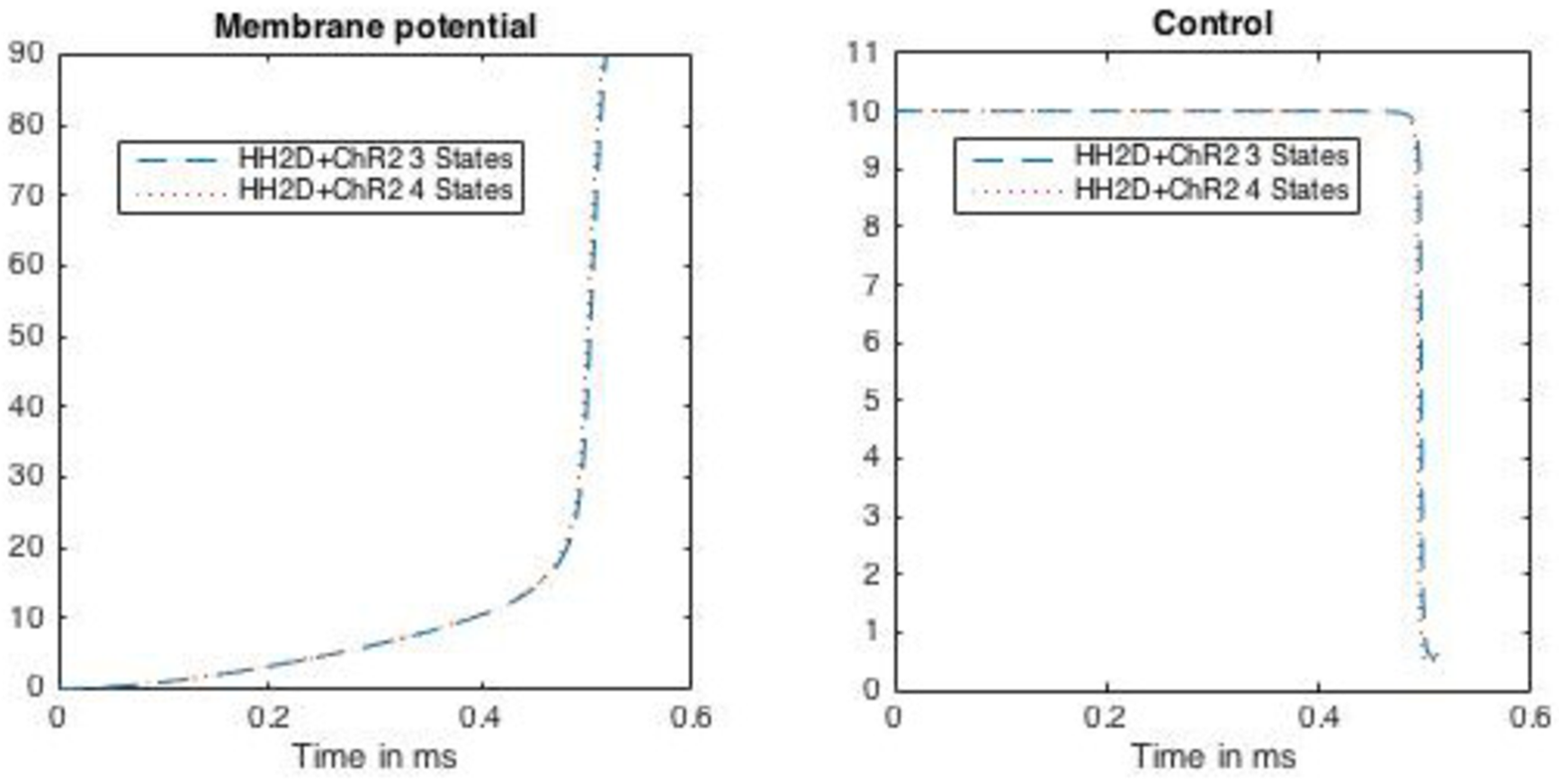}
\end{tabular}
\caption{Optimal trajectory and bang-bang optimal control for the HH2D-ChR2-3-states and HH2D-ChR2-4-states models with $u_{max}=$ a) $0.028$, b) $0.1$, c) $1$, d) $10\,\mathrm{ms}^{-1}$.}
\label{HH2D_ChR2_3and4States}
\end{figure}

%\begin{figure}[!ht]
%\begin{center}
%\includegraphics[scale=0.55]{HH2D_ChR2_4States2}
%\end{center}
%\caption{Optimal trajectory and bang-bang optimal control for the HH2D-ChR2-4-states model.}
%\label{HH2DCHR2_4States}
%\end{figure}
%
%\begin{figure}[!ht]
%\begin{center}
%\includegraphics[scale=0.6]{HH2D_ChR2_4States_Sing2}
%\end{center}
%\caption{Optimal trajectory and singular optimal control for the HH2D-ChR2-4-states model.}
%\label{HH2DCHR2_4States_Sing}
%\end{figure}

\subsection{The complete Hodgkin-Huxley model}

\paragraph{The ChR2-3-states model}\hspace*{0pt}\\

The complete Hodgkin-Huxley model is more difficult to analyze mathematically, and optimal singular controls cannot be excluded a priori as for the previous models. Nevertheless, singular controls do not appear in our numerical simulations. Figure \ref{HHCHR2_3States} shows the optimal trajectory and control for numerical values taken in Appendices \ref{AppendixHH} and \ref{AppendixChR2_3States}.

\begin{figure}[!ht]
\begin{center}
\includegraphics[scale=0.6]{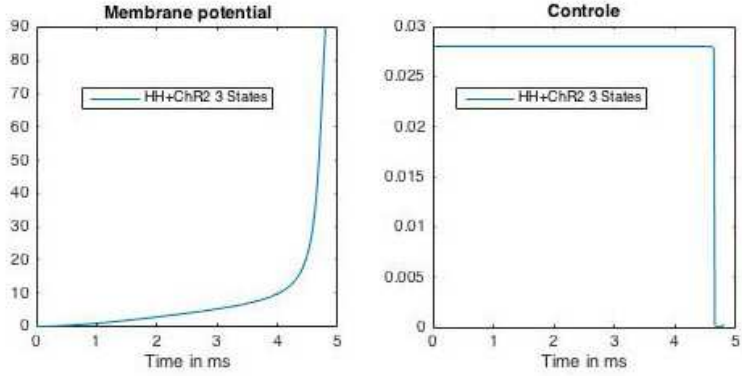}
\end{center}
\caption{Optimal trajectory and bang-bang optimal control for the HH-ChR2-3-states model.}
\label{HHCHR2_3States}
\end{figure}

\paragraph{The ChR2-4-states model}\hspace*{0pt}\\

We observe the same phenomenon than for the reduced Hodgkin-Huxley model, that is, for small values of $u_{max}$, the ChR2-4-states model slightly outperforms the ChR2-3-states model and when $u_{max}$ increases, both models match (Figure \ref{HH_ChR2_3and4States}). This constitutes a new argument in favor of the reduced Hodgkin-Huxley model since it captures the features of the complete model in terms of optimal control. Finally, the fact that both Hodgkin-Huxley models have almost the same behavior for the two ChR2 models means that they can be qualified as robust with regards to the mathematical modeling of ChR2.

\clearpage

\begin{figure}[!ht]
\centering
\begin{tabular}{c}
   a)\includegraphics[scale=0.54]{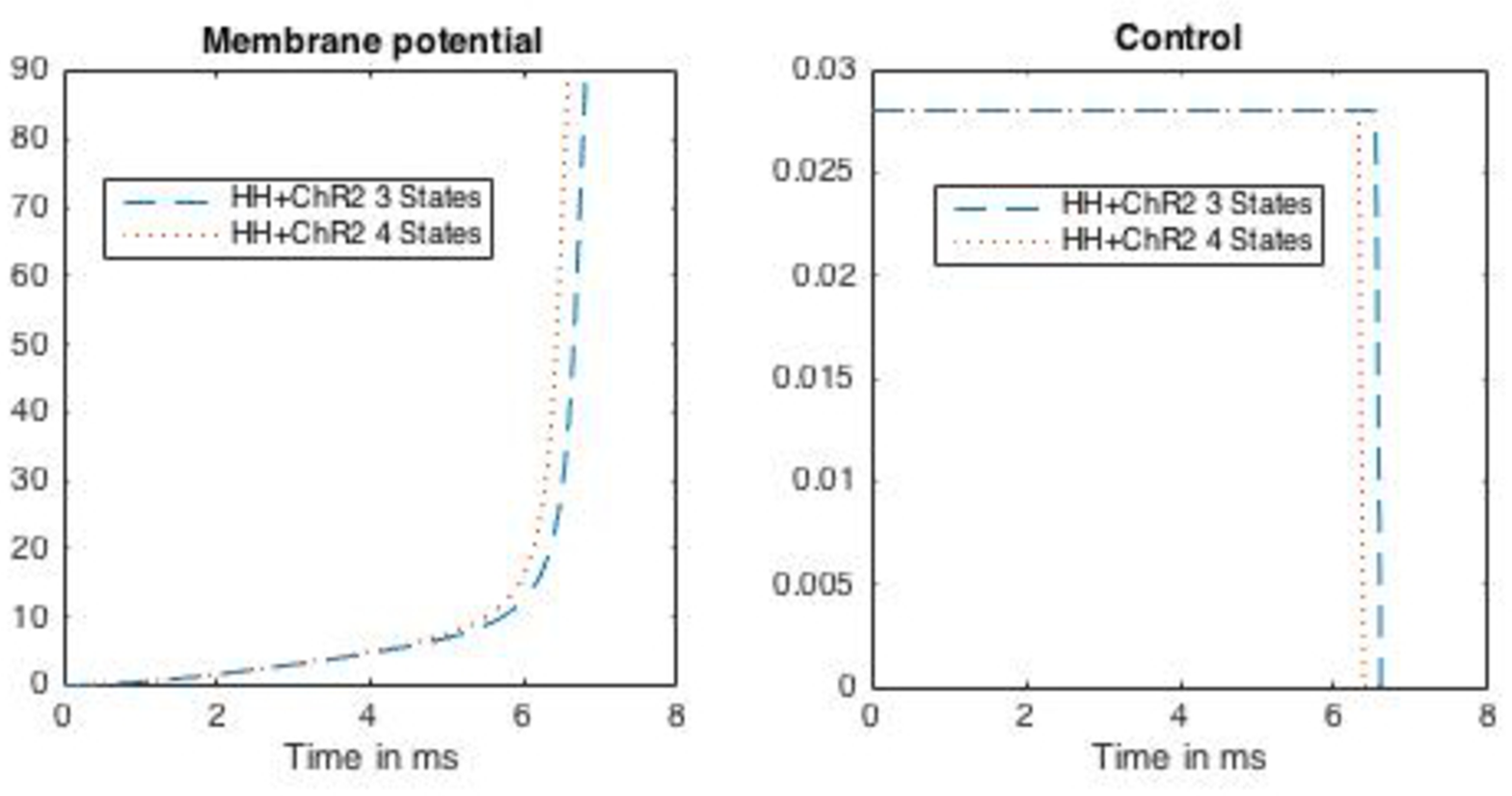} \\
   b)\includegraphics[scale=0.51]{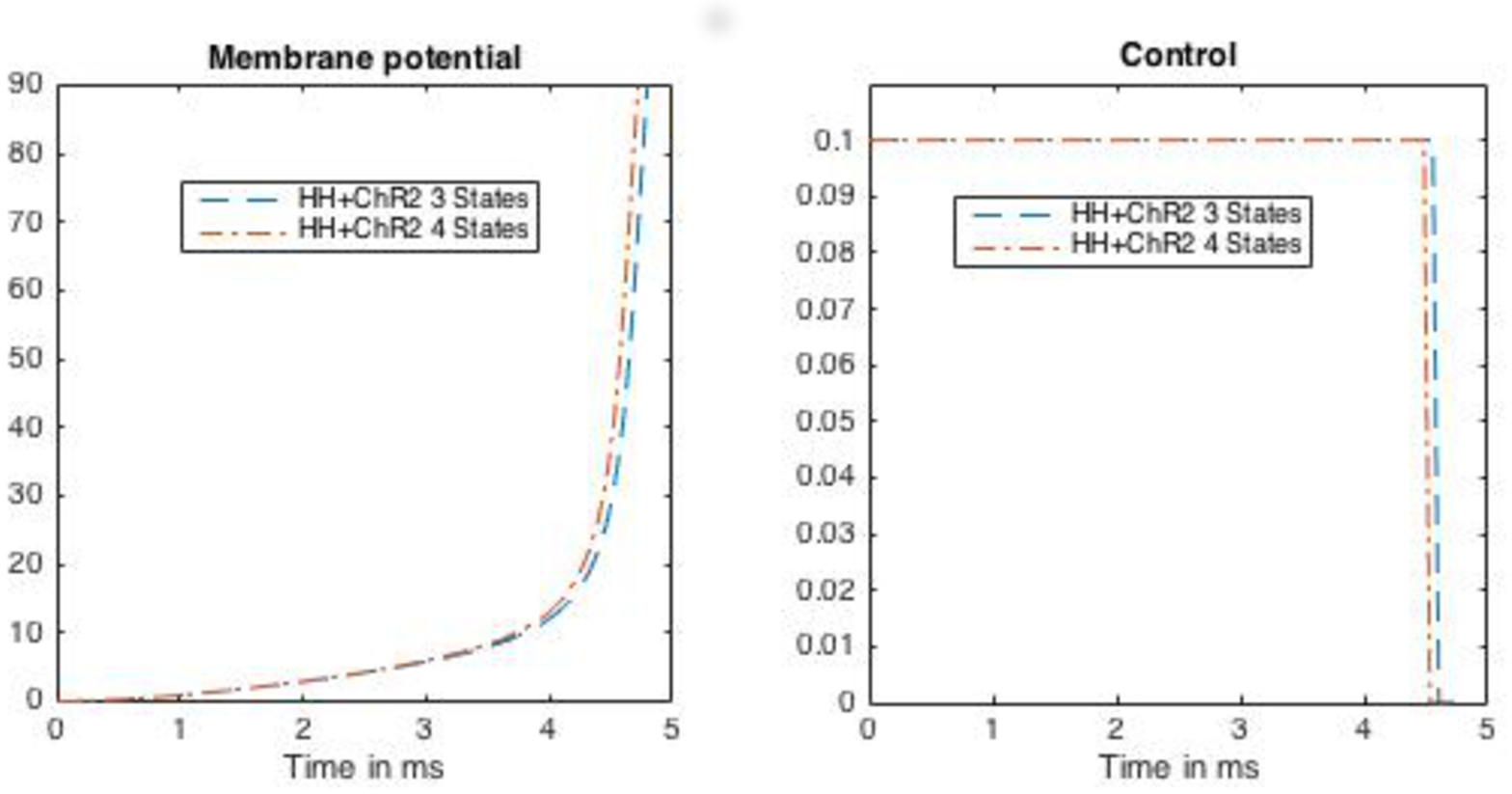} \\
    c)\includegraphics[scale=0.51]{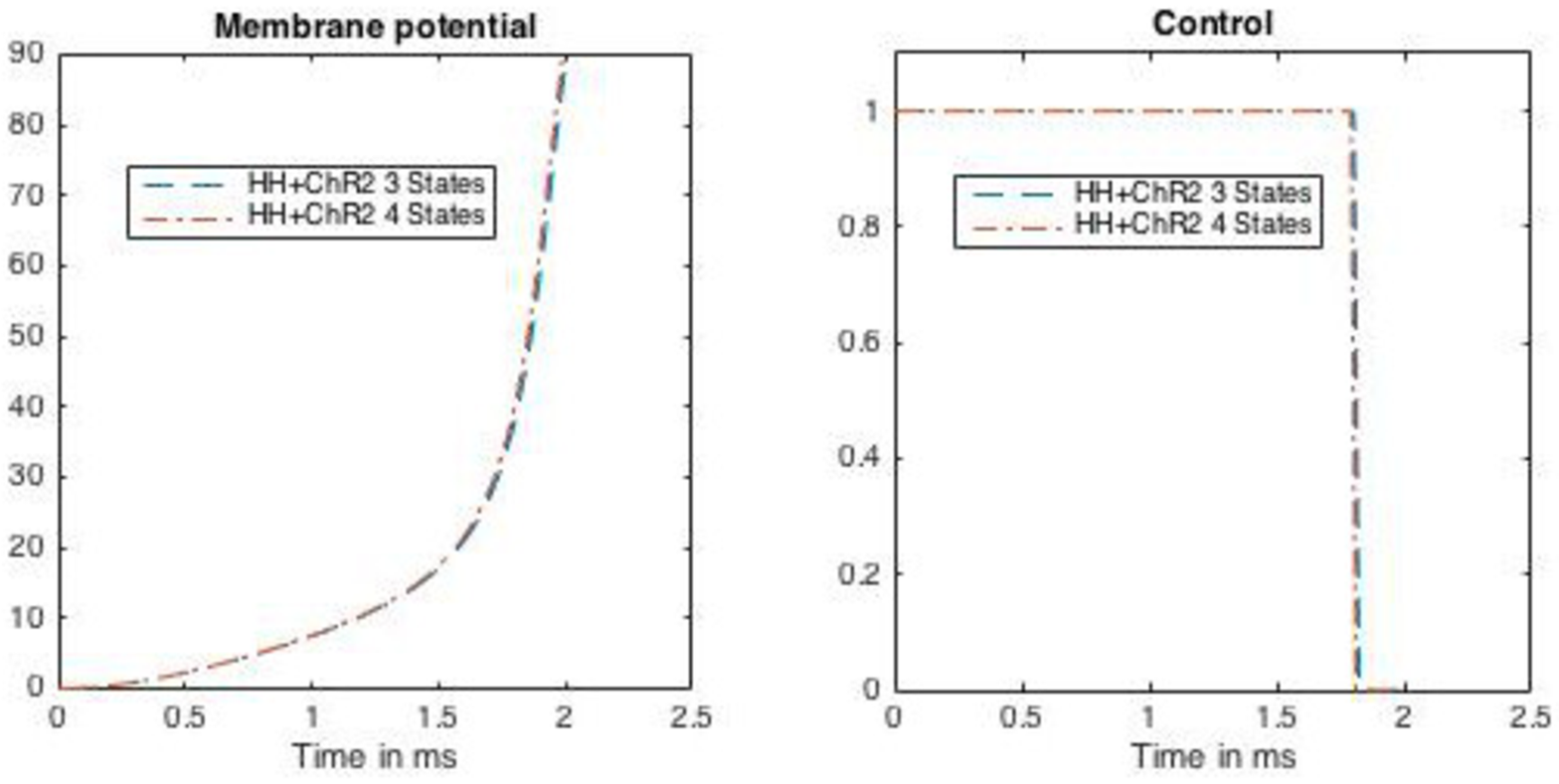} \\
    d)\includegraphics[scale=0.51]{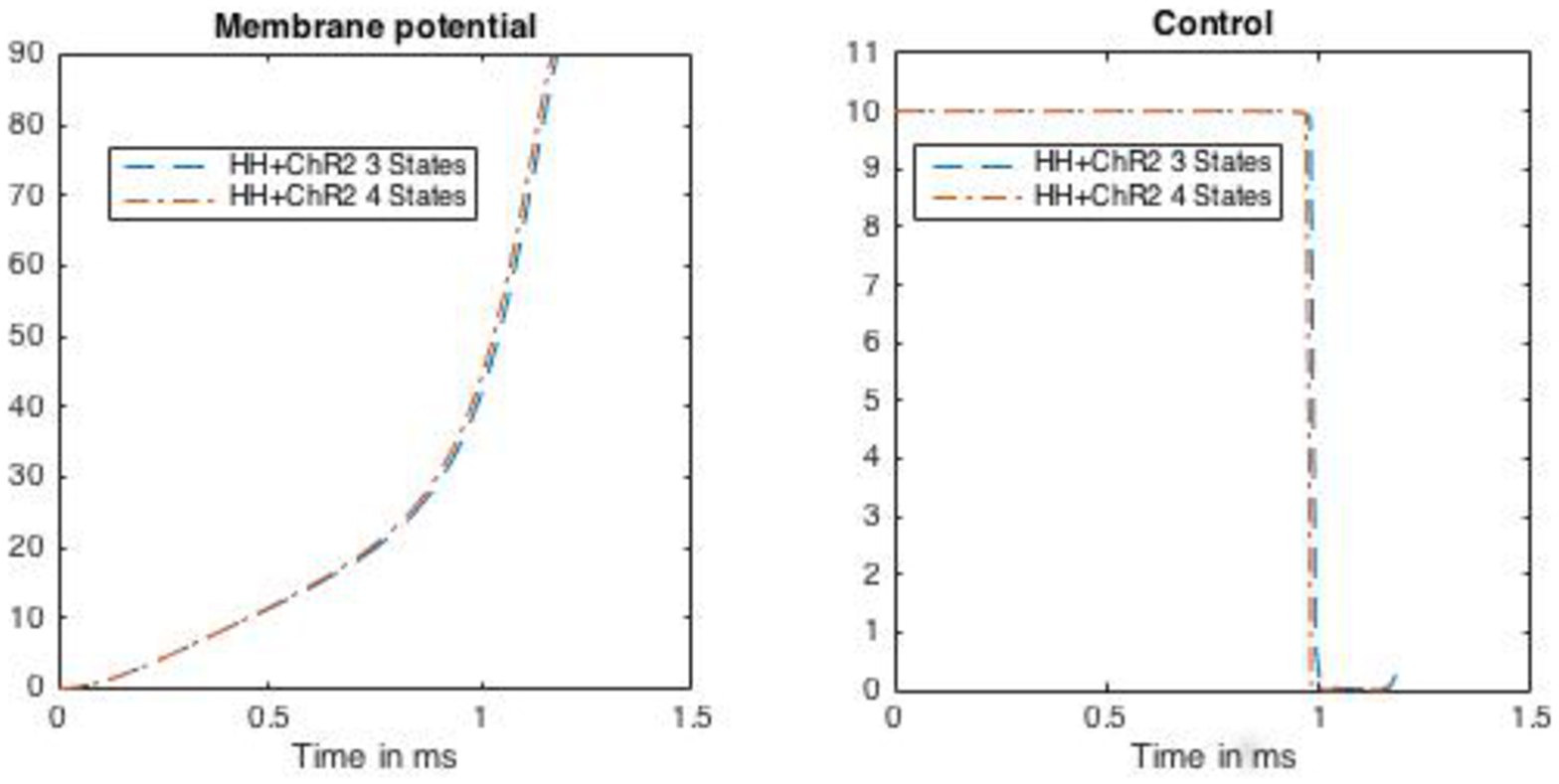}
\end{tabular}
\caption{Optimal trajectory and bang-bang optimal control for the HH-ChR2-3-states and HH-ChR2-4-states models with $u_{max}=$ a) $0.028$, b) $0.1$, c) $1$, d) $10\,\mathrm{ms}^{-1}$.}
\label{HH_ChR2_3and4States}
\end{figure}

%\begin{figure}[!ht]
%\begin{center}
%\includegraphics[scale=0.6]{HH_ChR2_4States2}
%\end{center}
%\caption{Optimal trajectory and banb-bang optimal control for the HH-ChR2-4-states model.}
%\label{HHCHR2_4States}
%\end{figure}
%
%\begin{figure}[!ht]
%\begin{center}
%\includegraphics[scale=0.6]{HH_ChR2_4States_Sing2}
%\end{center}
%\caption{Optimal trajectory and singular optimal control for the HH-ChR2-4-states model.}
%\label{HHCHR2_4States_Sing}
%\end{figure}

\subsection{Conclusions on the numerical results}

We begin with comments on the two versions of the ChR2 models for each neuron model. For every neuron model that we numerically treat, the ChR2-3-states and the ChR2-4-states versions behave qualitatively the same. We observe no optimal singular controls and the shapes of optimal controls and optimal trajectories are similar. Nevertheless, we can note some distinctions between the neuron models. For the FitzHugh-Nagumo model, the ChR2-4-states version always outperforms the ChR2-3-states version. This is also the case for the two Hodgkin-Huxley models with the important difference that, when the control maximal value increases, the optimal trajectory and optimal control quantitatively match. The Hodgkin-Huxley models are thus very robust with respect to the ChR2 modeling. The Morris-Lecar model displays an unusual behavior when we compare the ChR2-3-states and the ChR2-4-states versions. Indeed, for low values of the control maximal value, including the physiological value computed in Appendix \ref{AppendixChR2_3States}, the ChR2-3-states version outperforms the ChR2-4-states version and the opposite happens when the control maximal value increases.

As announced at the beginning of Section \ref{SectionApplication}, the numerical results invite to distinguish between two main behavior of neuron models with respect to optogenetic control. Most of the models, that is all the models except the Morris-Lecar, behave as physiologically expected. The optimal control is bang-bang, begins with a maximal arc, and has at most one switch. The Morris-Lecar model has more than one switch. This means that it is more efficient to switch on and off the light several times than just keep the light on almost all the way up to the spike. That is why we qualify this model as nonphysiological. Moreover, by only changing the value of the ChR2 equilibrium potential ($V_{ChR2}$) we can observe a change of the number of switches. Finally, the behavior of the Morris-Lecar model emphasizes the critical importance of optimal control since it allows to find a control that triggers a spike when the expected physiological stimulation (with at most one switch) fails to trigger a spike.

\begin{appendices}

\section{Numerical constants for the Morris-Lecar model}\label{AppendixML}
The numerical values of the several constants and their physiological meaning are taken from \cite{DitlevsenLIF} and gathered in Table \ref{NumericalConstantsML}.
 
\begin{table}[ht]
\caption{Meaning and numerical values of the constants appearing in the Morris-Lecar model}
\begin{center}
\begin{tabular}{lrl}
\hline
$V_1 =$ & $-1.2$ mV & Fitting parameter\\
$V_2 =$ & $18$ mV & Fitting parameter\\
$V_3 =$ & $2$ mV & Fitting parameter\\
$V_4 =$ & $30$ mV & Fitting parameter\\
$g_{Ca}=$ & $4.4$ $\mu\mathrm{ S/cm^2}$ & Maximal conductance of $Ca^{2+}$ channels\\
$g_{K}=$ & $8$ $\mu\mathrm{ S/cm^2}$ &  Maximal conductance of $K^{+}$ channels\\
$g_{L}=$ & $2$ $\mu\mathrm{ S/cm^2}$ & Conductance associated with the leakage current\\
$V_{Ca} =$ & $120$ mV & Equilibrium potential of $Ca^{2+}$ ions\\
$V_{K} =$ & $-84$ mV &  Equilibrium potential of $K^{+}$ ions\\
$V_{L} =$ & $-60$ mV & Equilibrium potential for the leak current\\
$C=$& $20$ $\mu\mathrm{ F/cm^2}$& Membrane capacitance\\
$\phi =$&$0.04$ $\mathrm{ms^{-1}}$& Fitting parameter\\
\hline
\end{tabular}
\end{center}
\label{NumericalConstantsML}
\end{table}

\clearpage

\noindent Table \ref{LongtinChR2} gathers the numerical values for Figure \ref{MLCHR2_3StatesNoSpike}.

\begin{table}[ht]
\caption{Meaning and numerical values of the constants, taking from \cite{Longtin}, appearing in the Morris-Lecar model}
\begin{center}
\begin{tabular}{lrl}
\hline
$V_1 =$ & $-0.01$ mV & Fitting parameter\\
$V_2 =$ & $0.15$ mV & Fitting parameter\\
$V_3 =$ & $0.1$ mV & Fitting parameter\\
$V_4 =$ & $0.145$ mV & Fitting parameter\\
$g_{Ca}=$ & $1.0$ $\mu\mathrm{ S/cm^2}$ & Maximal conductance of $Ca^{2+}$ channels\\
$g_{K}=$ & $2.0$ $\mu\mathrm{ S/cm^2}$ &  Maximal conductance of $K^{+}$ channels\\
$g_{L}=$ & $0.5$ $\mu\mathrm{ S/cm^2}$ & Conductance associated with the leakage current\\
$V_{Ca} =$ & $1.0$ mV & Equilibrium potential of $Ca^{2+}$ ions\\
$V_{K} =$ & $-0.7$ mV &  Equilibrium potential of $K^{+}$ ions\\
$V_{L} =$ & $-0.5$ mV & Equilibrium potential for the leak current\\
$C=$& $1.0$ $\mu\mathrm{ F/cm^2}$& Membrane capacitance\\
$\phi =$&$0.333$ $\mathrm{ms^{-1}}$& Fitting parameter\\
\hline
\end{tabular}
\end{center}
\label{LongtinChR2}
\end{table}

\section{Numerical constants for the Hodgkin-Huxley model}\label{AppendixHH}
\begin{align*}
\alpha_n(V) &= \frac{0.1-0.01V}{e^{1-0.1V}-1},&\beta_n(V) &= 0.125e^{-\frac{V}{80}},\\
\alpha_m(V) &= \frac{2.5-0.1V}{e^{2.5-0.1V}-1}, &\beta_m(V) &= 4e^{-\frac{V}{18}},\\
\alpha_h(V) &= 0.07e^{-\frac{V}{20}},&\beta_h(V) &= \frac{1}{e^{3-0.1V}+1}.
\end{align*}

The following table gathers the numerical values of the Hodgkin-Huxley model, as given in the original paper \cite{HH}.
\begin{table}[!ht]
\caption{Meaning and numerical values of the constants appearing in the Hodgkin-Huxley model}
\begin{center}
\begin{tabular}{lrl}
\hline
$\bar{g}_{K}=$ & $36$ $\mu\mathrm{ S/cm^2}$ &  Maximal conductance of $K^{+}$ channels\\
$\bar{g}_{Na}=$ & $120$ $\mu\mathrm{ S/cm^2}$ & Maximal conductance of $Na^{2+}$ channels\\
$g_{L}=$ & $0.3$ $\mu\mathrm{ S/cm^2}$ & Conductance associated with the leakage current\\
$E_{Na} =$ & $115$ mV & Equilibrium potential of $Na^{2+}$ ions\\
$E_{K} =$ & $-12$ mV &  Equilibrium potential of $K^{+}$ ions\\
$E_{L} =$ & $-10.6$ mV & Equilibrium potential for the leak current\\
$C=$& $0.9$ $\mu\mathrm{ F/cm^2}$& Membrane capacitance\\
\hline
\end{tabular}
\end{center}
\label{NumericalConstantsHH}
\end{table}

The equilibrium potential $E_L$ of the leakage current is usually set so that the equilibrium value of the (HH) system is such that $V=0$.

\section{Numerical constants for the ChR2 models}\label{NumericalConstantsChR2}
\subsection{The 3-states model}\label{AppendixChR2_3States}
The constants of the model are the rates $K_d$ and $K_r$ of the transitions between the open state and the light adapted closed state and between the two closed states, the maximal conductance $g_{ChR2}$ and the equilibrium potential $V_{ChR2}$. As specified in Section \ref{resultsSection}, we assume that these rates are constants during the evolution in order to obtain an affine control system. For the numerical computations, we took the values given in Table 1 of \cite{Nikolic}:

\[K_d = 0.2\text{ ms}^{-1}, \qquad K_r = 0.021 \text{ ms}^{-1}.\]

The maximal conductance is given by the formula $g_{ChR2} = \rho_{ChR2}g^*_{ChR2}$, with $\rho_{ChR2}$ the density of channels and $g^*_{ChR2}$ the conductance of a single channel. These values are taken from \cite{Foutz} to obtain

\[
g_{ChR2} = 0.65 \text{ mS}\cdot \text{cm}^{-2}.
\]

As mentioned right after in Appendix \ref{AppendixHH}, the physiological equilibrium membrane potential is mathematically shifted to equal 0. The equilibrium potential of the $ChR2$ that is usually measured around 0 (\cite{Foutz}) and very often taken as 0 (\cite{Foutz},\cite{Nikolic}). The exact value 0 would raise a mathematical problem because since we shifted the value of $E_L$ so that $V=0$ corresponds to the equilibrium point of the uncontrolled system we start from. Indeed, $V=0$ would also correspond to an equilibrium point of the controlled system, regardless of the value of the control. For this reason, we shifted the value of $V_{ChR2}$ and took it equal to $60$mV. This value corresponds to the shift of the membrane resting potential for the Morris-Lecar and Hodgkin-Huxley models.

Finally we can give an estimation of the physiological maximal value $u_{max}$ of the control. Indeed, upon illumination, the transition rate between the dark adapted closed state and the open state in \cite{Nikolic} is $\varepsilon F$ where $\varepsilon =0.5$ is the quantum efficiency and $F$ is given by the formula

\[
 F = \frac{\sigma_{ret} \phi }{w_{loss}},
\]  

where $\sigma_{ret} \simeq 10^{-8} \mu\text{m}^2$ is the retinal cross section (cross section of the photon receptor on the $ChR2$), $\phi = 6.2\times10^{9}\text{ ph}\cdot\mu\text{m}^{-2}\cdot\text{s}^{-1}$ is the original flux of photons and $w_{loss} =1.1$ is a loss factor. As for the numerical value of $K_d$ and $K_r$ we took the one of Table 1 in  \cite{Nikolic} for the value of $\phi$. With these values we get 

\[
u_{max} = 0.028 \text{ ms}^{-1}.
\]

\subsection{The 4-states model}\label{AppendixChR2_4States}

The numerical values for the ChR2-4-states model are taken from \cite{Foutz} and gathered in Table \ref{TableChR2_4States} below

\begin{table}[ht]
\caption{Numerical values of the constants appearing in the ChR2-4-States model}
\begin{center}
\begin{tabular}{lcl}
\hline
$K_{d1} =$ & $0.13$ $\mathrm{ms}^{-1}$ & Decay rate\\
$K_{d2} =$ & $0.025$ $\mathrm{ms}^{-1}$ & Decay rate\\
$e_{12} =$ & $0.053$ $\mathrm{ms}^{-1}$  & Transition rate\\
$e_{21} =$ & $0.023$ $\mathrm{ms}^{-1}$ & Transition rate\\
$K_r=$ & $0.004$ $\mathrm{ms}^{-1}$ & Recovery rate\\
$\varepsilon1=$ & $0.5$ & Quantum efficency for $o_1$\\
$\varepsilon2=$ & $0.1$ & Quantum efficency for $o_2$\\
$g_1 =$ & $50$ $\mathrm{fS}$  & $o_1$ state conductance\\
$\rho=$ & $0.05$ & Relative conductance of the open states\\
$\rho^*_{ChR2} =$ & $130$ $\mu\mathrm{m}^{-2}$ &  ChR2 density\\ 
$g_{ChR2} =$ & $0.65$ $\mathrm{mS}\cdot \mathrm{cm}^{-2}$ & $ChR2$ maximal conductance\\
\hline
\end{tabular}
\end{center}
\label{TableChR2_4States}
\end{table}

\end{appendices}

\nocite{*}
\bibliographystyle{plain}
\bibliography{Minimal_Time}

\end{document}